\documentclass{article}

\usepackage{amssymb}
\usepackage{amsmath}
\usepackage{amsthm}
\usepackage{srcltx}
\usepackage{fleqn}
\usepackage{authblk}
\usepackage{txfonts}
\usepackage{bbold}
\usepackage[misc,geometry]{ifsym}

\usepackage{enumitem}
\usepackage{xcolor}

\newcommand{\nat}       {{\rm I\!N}}

\newcommand{\der}       {\vartriangleright}
\newcommand{\nder}       {\blacktriangleright}

\newcommand{\set}        {\textsc{Set}}
\newcommand{\for}        {\textsc{Fmla}}
\newcommand{\setset}        {{\set\times\set}}
\newcommand{\setfor}        {{\set\times\for}}
\newcommand{\mini}        {\mathsf{min}}
\newcommand{\maxi}        {\mathsf{max}}

\newcommand{\val}        {\mathsf{Val}}
\newcommand{\ty}        {\mathsf{Th}}
\newcommand{\CR}        {\mathsf{CRL}}

\newcommand{\co}        {\mathcal{C}}

\newcommand{\condCM}[1]        {\textsf{CM}{#1}}
\newcommand{\condCRT}[1]        {\textsf{CR$_\TCon$}{#1}}
\newcommand{\condCRS}[1]        {\textsf{CR$_\SCon$}{#1}}
\newcommand{\condCOT}[1]        {\textsf{CO$_\TCon$}{#1}}
\newcommand{\condCOS}[1]        {\textsf{CO$_\SCon$}{#1}}

\newcommand{\one}        {\mathbb{1}}
\newcommand{\zero}        {\mathbb{0}}
\newcommand{\lang}        {\mathbb{L}}

\newcommand{\disjS}[2]		{#1\cap#2=\varnothing}
\newcommand{\ndisjS}[2]		{#1\cap#2\neq\varnothing}
\newcommand{\SCon}        {\mathsf{S}}
\newcommand{\TCon}        {\mathsf{T}}
\newcommand{\derSCon}        {\der^\SCon}
\newcommand{\derTCon}        {\der^\TCon}
\newcommand{\biv}        {\nu}
\newcommand{\Sem}        {\mathbf{V}}
\newcommand{\ComR}        {\textsf{CML}}
\newcommand{\CSem}        {\textsf{SML}}
\newcommand{\CoL}        {\textsf{COL}}
\newcommand{\fini}        {\textsf{F}}

\newcommand{\Val}[1]        {\Sem_{#1}}
\newcommand{\Nder}[1]        {\nder_{#1}}
\newcommand{\Der}[1]        {\der_{#1}}
\newcommand{\bbtop}	{\top\!\!\!\top}
\newcommand{\derMini}[1]	{{\derSCon_{#1}}[\mini]}
\newcommand{\derMaxi}[1]	{{\derSCon_{#1}}[\maxi]}
\newcommand{\imp}	{\to}
\newcommand{\At}	{P}

\newcommand{\coS}	{\co2}
\newcommand{\coA}[2]	{\co_{#1}^{\mathsf{A}{:}#2}}
\newcommand{\coD}[2]	{\co_{#1}^{\mathsf{D}{:}#2}}

\newtheorem{thm}{Theorem}[section]

\newtheorem{lem}[thm]{Lemma}
\newtheorem{prop}[thm]{Proposition}

\numberwithin{equation}{section}

\begin{document}

\title{
What is a logical theory?\\ On theories containing assertions and denials%
\thanks{%
The second author acknowledges that the work was done under the scope of Project UID~/ EEA~/ 50008~/ 2019 of Instituto de Telecomunica\c{c}\~oes, financed by the applicable framework (FCT~/ MEC through national funds and co-funded by FEDER-PT2020). 
The third author acknowledges partial funding by CNPq. 
The first author passed away on 26 Aug 2017.}
}


\author[1]{Carolina~Blasio}
\author[2,3]{Carlos~Caleiro}
\author[4]{Jo\~ao Marcos}
\affil[1]{Independent scholar}
\affil[2]{SQIG--Instituto de Telecomunica\c{c}\~{o}es, Portugal} 
\affil[3]{Dep. Matem\'atica--Instituto Superior T\'ecnico, Universidade de Lisboa, Portugal}
\affil[4]{LoLITA--DIMAp, UFRN, Brazil}

\date{}

\maketitle

\begin{abstract}
The standard notion of formal theory, in Logic, is in general biased exclusively towards assertion: it commonly refers only to collections of assertions that any agent who accepts the generating axioms of the theory should also be committed to accept.
In reviewing the main abstract approaches to the study of logical consequence, we point out why this notion of theory is unsatisfactory at multiple levels, and introduce a novel notion of theory that attacks the shortcomings of the received notion by allowing one to take both assertions and denials on a par.
This novel notion of theory is based on a bilateralist approach to consequence operators, which we hereby introduce, and whose main properties we investigate in the present paper.
\end{abstract}


\section*{Towards a widened notion of theory}

When thinking of a logical \textit{theory} as some collection of judgments closed under a given notion of consequence, it is often useful to adopt the proof-theoretic frame of mind and think of a collection of axioms and rules together with the theorems that may be derived from them, or else to put the glasses of the formal semanticist and look at the collection of sentences entailed from the initially given collection.  A general abstract approach underlying both proof theory and formal semantics, though, may be undertaken by way of the theory of consequence operators.  In the standard approach promoted by Alfred Tarski, a theory is nothing but a fixed point of a certain given \textit{consequence operator}: in the usual case, one talks about a collection of \textit{asserted} sentences that contains all other assertions that may be thought of as `following from' that very collection; dually, one might also talk about of a collection of \textit{denied} sentences that contains all other denials that follow from it.

If one takes for granted that a denial consists simply in the complement of an assertion, as much of the literature on Logic would seem to do (notice though that this was already put into question as early as in~\cite{Curry:FML}), the local reading of the rule of \textit{modus ponens} as constraining one to assert~$B$ whenever one commits oneself to asserting both $A$-implies-$B$ and~$A$ would be equivalent to reading that very rule in the form of \textit{modus tollens}, constraining one to denying~$A$ whenever one commits to asserting $A$-implies-$B$ while denying~$B$.  Further, if the \textit{proof} of an assertion amounts to globally recognizing its \textit{validity}, it would seem wrong to understand the \textit{counterproof} of a denial as a recognition of its invalidity; the latter, and dual, attitude would rather amount to a recognition that a denial corresponds to an \textit{unsatisfiable} judgment.  Assuming that unsatisfiability is the complement of validity would however consist in disposing of what would arguably constitute the most interesting class of logical sentences, namely that of contingencies: indeed, only the contingent sentences are really `informative' for the logician, in the sense that they always make a difference in terms of therewith associated states of affairs (viz.\ models), when added to either side of an inferential statement.

Thinking about assertions and denials both as first-class citizens in the activity of analyzing inferences, by considering theories that are presented by way of both kinds of judgments, not only allows the above mentioned unexpected asymmetries to be fixed, but it actually gives one the opportunity to construct strictly more expressive inferential statements than one would be able to construct using only one of the judgments of asserting or of denying.  Here goes, by way of an informal example, a theory that has no finite assertion-based presentation, but that does have a finite presentation in terms of a denial-based presentation.
Consider a set~$\lang$ of sentences represented by the natural numbers and an extra sentence~$H$.  Suppose that a formal semantics for the underlying logic takes~$H$ to be false only if~$n$ is true for every natural number~$n$, and makes no further assumptions.  From the inferential viewpoint, as we shall see, this corresponds to committing to the consecutions written as $\varnothing\der \{n,H\}$, for every $n\in\nat$ (the comma on the right-hand side of the consequence relation symbol~$\der$ could of course be eliminated in the presence of a disjunction in the underlying object-language).  Where $\Gamma_1,\Gamma_0\subseteq\lang$, let's say that $(\Gamma_1,\Gamma_0)$ presents a theory in which the sentences in~$\Gamma_1$ are taken as asserted and those of~$\Gamma_0$ are taken as denied.  Then, as the reader will be able to check by himself after the next couple of sections of the present paper, for every finite $\Gamma\subsetneq\nat$ the `closure' of the purely assertion-based presentation $(\Gamma,\varnothing)$ is the theory $(\Gamma,\varnothing)$ itself; on the other hand, the infinite theory $(\nat,\{H\})$ is the closure of the purely denial-based finite presentation $(\varnothing,\{H\})$.

The latter example properly fits within the enterprise of investigating abstract consequence relations in terms of a framework often called 
`$\setset$ consequence', `symmetric[al] consequence' or `generalized consequence' (cf.\ \cite{ShoesmithSmiley,Gabbay:SIiHIL,Segerberg,DunnHardegree,Zygmunt,Humberstone:Connectives,HumberstoneScott}, and others), and contrasts with the traditional framework known as 
`$\setfor$ consequence', which allows for multiple premises and focuses on a single conclusion at any given time. 
For very natural reasons, as we shall see, the investigation of logical theories containing both assertions and denials, in the present paper, will be based on the above mentioned generalized notion of consequence.

We start Section~\ref{sec:consequence} of the present paper by stating the properties of a somewhat intuitive notion of `compatibility relation' among sets of sentences, taking next the generalized notion of consequence relation (that we call `$\SCon$-consequence relation') to be its complement and the Tarskian-inspired notion of consequence relation (that we call `$\TCon$-consequence relation') to impose a structural restriction on $\SCon$-consequence.
We rephrase then the $\TCon$-consequence relations in terms of `$\TCon$-consequence operators', define the space of all theories associated to $\TCon$-consequence, and consider some of the numerous possible $\TCon$-consequence counterparts of a given $\SCon$-consequence relation.  This section also recalls the meaning of the property of `finitariness' (of which the so-called `compactness property' constitutes the semantic rendering), as applied to the above mentioned notions.  The role of the latter property will be taken into account in later sections.%

Section~\ref{sec:semantics} clarifies how compatibility relations are connected, among other things, to purely meta-theoretical formulations of the fundamental logical principles of Excluded Middle and Non-Contradiction.  In this section we recall and discuss the Galois connections that may be established between formal semantics and the notions of compatibility and consequence, and we highlight the well-known problem concerning the failure of `absoluteness', according to which there is in general no one-to-one correspondence between $\TCon$-consequence relations and the therewith associated collections of bivaluations.  Together with some other illustrations provided in this section (e.g., the example of a non-fini\-tary $\SCon$-consequence relation that induces a finitary $\TCon$-consequence relation), the latter expressive shortcoming is used to motivate our present quest for an adequate generalization of consequence operators and an adequate notion of theory that fits the bill in the study of $\SCon$-consequence relations.

In Section~\ref{sec:genConsOp} we investigate more deeply the class of $\TCon$-consequence relations and operators that could be associated to the same given $\SCon$-consequence relation.  To put it in deductive terms, the main idea here is to consider theories presented in terms of both axioms and anti-axioms, and to allow one to ask for their associated theorems and anti-theorems: namely, for a given $\SCon$-consequence relation~$\der$, for each fixed set~$\Gamma_0$ of background denied assumptions 
we show how to define a $\TCon$-consequence operator $\coD{\der}{\Gamma_0}$, and for each fixed set~$\Gamma_1$ of background asserted assumptions 
we show how to define a $\TCon$-consequence operator $\coA{\der}{\Gamma_1}$.  The latter consequence relations are then used to explain the behavior of our novel bilateralist notion of $\SCon$-consequence operator, according to which the $\SCon$-closure $\coS(\Gamma_1,\Gamma_0)$ of the pair $(\Gamma_1,\Gamma_0)$ is given by the `theory-pair' $(\coD{\der}{\Gamma_0}(\Gamma_1),\coA{\der}{\Gamma_1}(\Gamma_0))$.  One should observe, in particular, how each component of the latter notion has its own associated context of judgment: one is to talk accordingly about the theorems that are concluded `modulo a set of denials', and about the anti-theorems that are concluded `modulo a set of assertions'.

To the best of our knowledge, all other extant notions of theory-pair proposed in the literature are based on partitions of the whole set of sentences into those that are asserted and those that are denied.  This corresponds in fact to (consistent and) complete, or maximal theories, and an appropriate adequate semantics may be associated to the latter, for any given logic.  However, as we argue in this paper, there is no realistic reason why theories should be so `fully informative': in general theories need not divide all sentences into those belonging to the class of theorems and those belonging to the class of anti-theorems.
At last, in Section~\ref{sec:theories}, we proceed to show how the above mentioned notion of $\SCon$-closure may accordingly be used to provide a novel notion of theory (as a fixed point of~$\coS$) that takes both assertions and denials on an equal footing, and that allows us to investigate the associated spaces of theories.  
We finish the paper by generalizing a well-known result that identifies the property of the space of theories that corresponds to finitariness of consequence.

\paragraph{Historical digression.}
The terminology introduced by Tarski in~\cite{Tarski:fund30} 
---and used in fact by most of the subsequent Polish literature on the study of logical calculi--- 
for the fixed points of a given closure operator~$C$ was that of
\emph{deductive system} or, frequently and more specifically,
(\emph{closed}) \emph{$C$-system}.  Much later on
(cf.~\cite{Wojcicki:TLC}), \emph{logical theories} were to be
identified with collections of said deductive systems that happened to be closed
under substitutions (cf.~\cite{LosSuszko58}).  Tarski's initial aim
was that of studying in abstract the fundamental properties and
underlying concepts of Hilbert's axiomatic method in metamathematics
(to which Tarski referred as `the methodology of the deductive
sciences').
Such abstract ---and axiomatic--- approach to consequence aimed thus
at generalizing the concrete approach to the (proof-theoretical)
formalization of mathematical theories based on collections of
inference rules of a certain format, the closure under which would
suffice to turn a given collection of assertions into the theory
thereby induced.  Some time later, however, Tarski claimed that there was
a mismatch between the proof-theoretical approach and the
`common concept of consequence' (cf.~\cite{Tarski:36}), namely, the
one that was to be captured by `scientific semantics' ---a programme
inspired by Carnap's early model theory.   Accordingly, Tarski
explicitly formulated then the \emph{truth-preserving} concept of
consequence which is the one nowadays canonically associated with
\mbox{$\TCon$-consequence relations}.  It is worth noting that no connection 
was established at the time between sets of (object-level) axioms and classes of
models; general results concerning the adequacy of the semantical
notion of consequence to the abstract notion of consequence were
unheard of in this early period.

We also note, as an aside, that the now standard notion of `$\TCon$-consequence' is \textit{not} the one that was originally entertained in~\cite{Tarski:fund30}, for Tarski always took finitariness for granted, as an integral part of his axiomatization of consequence operators.  Non-finitary consequence operators, indeed, were to become the standard only much later (cf.~\cite{LosSuszko58}).  A similar misattribution that persists in the literature concerns the notion of `$\SCon$-consequence', which was \textit{not} entertained by Scott in~\cite{Scott:MVL74-Tarski}, where a structural restriction on the cardinality of the consecutions was imposed to the effect that, once more, only finitary logics happened to be considered. The study of (non-finitary) generalized consequence relations, in the sense invoked here in talking about `$\SCon$-consequence', properly started with~\cite{ShoesmithSmiley}.





\section{Capturing the notion of logical consequence}
\label{sec:consequence}


Let $\lang$ be a non-empty set of \emph{sentences}.  
We will assume that the judgments of \emph{assertion} and \emph{denial} are primitive in our metalanguage, and in what follows 
we will intuitively think of the \emph{consecution} $(\Delta_1,\Delta_0)$ as a meta-logical expression concerning the `compatibility' of certain judgments, 
namely, the assertion of all sentences in~$\Delta_1\subseteq\lang$ and the denial of all sentences in~$\Delta_0\subseteq\lang$. 
%
Building on that idea, a (\emph{canonical logical}) \emph{
compatibility relation} (\emph{on $\lang$}) will be here defined as any relation~${\nder}$ on $\wp(\lang)\times\wp(\lang)$ satisfying, 
for every $\Pi,\Pi^\prime,\Sigma,\Sigma^\prime,\Delta\subseteq \lang$:
%
\begin{enumerate}[labelindent=\parindent, leftmargin=*,label=($\condCM{\arabic*}$)]\setcounter{enumi}{-1}
	\item
	\label{comp:E}
	if $\Pi^\prime\cup\Pi\nder\Sigma\cup\Sigma^\prime$, then $\Pi\nder\Sigma$ 
	\item 
	\label{comp:U}
	if $\Pi\nder\Sigma$, then $\disjS{\Pi}{\Sigma}$
	\item 
	\label{comp:A}
	if $\Pi\nder\Sigma$, then 
	there is some $\Delta^\prime\subseteq\Delta$ such that $\Delta^\prime\cup\Pi\nder\Sigma\cup
	(\Delta\setminus\Delta^\prime)$
\end{enumerate}
The reading of \ref{comp:E} is immediate: in any state of affairs in which a certain set of sentences~$\Delta_1=\Pi^\prime\cup\Pi$ is asserted while a certain set of sentences~$\Delta_0=\Sigma\cup\Sigma^\prime$ is denied, one may in particular say that all subsets of~$\Delta_1$ are asserted and that all subsets of~$\Delta_0$ are denied.  Furthermore, on the one hand, taking $\Pi=\Sigma=\{A\}$, property \ref{comp:U} says that the sentence~$A$ may not be simultaneously asserted and denied; on the other hand, taking $\Delta=\{A\}$, property \ref{comp:A} says that the sentence~$A$ must be either asserted or denied (in a context where the sentences in~$\Pi$ are asserted and those in~$\Sigma$ are denied).  One might say thus that \ref{comp:U} provides a meta-logical formulation of the `Principle of Non-Contradiction', and disallows for \textit{glutty} states of affairs in which a sentence is simultaneously asserted and denied: In any given (consistent) state of affairs, asserting a given sentence~$A$ should not be compatible with denying it. 
Dually, one might say that \ref{comp:A} provides a meta-logical formulation of the `Principle of Excluded Middle', and disallows for \textit{gappy} states of affairs in which a sentence is neither asserted nor denied: In no state of affairs can a sentence~$A$ fail to be either asserted or denied.

The complement~$\der$ of a 
compatibility relation~$\nder$ on~$\wp(\lang)\times\wp(\lang)$ will here be called an \emph{$\SCon$-consequence relation} (\emph{on~$\lang$}).  It should be clear that it satisfies the following properties, 
for every $\Pi,\Pi^\prime,\Sigma,\Sigma^\prime,\Delta\subseteq \lang$:
\begin{enumerate}[labelindent=\parindent, leftmargin=*,label=($\condCRS{\arabic*}$)]\setcounter{enumi}{-1}
	\item
	\label{SCon:E}
	if $\Pi\der\Sigma$, then $\Pi^\prime\cup\Pi\der\Sigma\cup\Sigma^\prime$ 
	\item 
	\label{SCon:U}
	if $\ndisjS{\Pi}{\Sigma}$, then $\Pi\der\Sigma$
	\item 
	\label{SCon:A}
	if $\Delta^\prime\cup\Pi\der\Sigma\cup(\Delta\setminus\Delta^\prime)$ for every $\Delta^\prime\subseteq\Delta$, then 
	$\Pi\der\Sigma$
\end{enumerate}
Examples of $\SCon$-consequence relations abound in the literature, but many logicians seem still not to be so familiar with them ---or even resist to them without justifications based on anything but misunderstandings about what they mean and how they behave at the meta-logical level.

If one constrains an $\SCon$-consequence relation into a relation on $\wp(\lang)\times\lang$, one will be said to define a \emph{$\TCon$-consequence relation} (\emph{on~$\lang$}). Note that a $\TCon$-consequence relation~$\der$ satisfies the following properties, for every $\Pi,\Pi^\prime,\Delta\subseteq\lang$ and every $A\in\lang$: 
\begin{enumerate}[labelindent=\parindent, leftmargin=*,label=($\condCRT{\arabic*}$)]\setcounter{enumi}{-1}
	\item
	\label{TCon:E}
	if $\Pi\der A$, then $\Pi^\prime\cup\Pi\der A$ 
	\item 
	\label{TCon:U}
	if $A\in{\Pi}$, then $\Pi\der A$
	\item 
	\label{TCon:A}
	if $\Delta\cup\Pi\der A$ and $\Pi\der\delta$ for every $\delta\in\Delta$, then 
	$\Pi\der A$
\end{enumerate}
%
In what follows, whenever we need to disambiguate between the symbols for $\SCon$-con\-se\-quence and for $\TCon$-con\-sequence, we shall write, respectively, $\derSCon$ and $\derTCon$.
For a se\-man\-ti\-cally-inspired analogy, one might think of the opposition between the notions of `compatibility' and `consequence' as reflecting the (dual) opposition between the notions of `satisfiability' and `validity'.

An alternative way of describing $\TCon$-consequence relations may be procured by way of a generalization of Kuratowski's axioms on topological closure, namely by omitting the axioms according to which $\co(\varnothing)=\varnothing$ and $\co(\Delta\cup\Delta^\prime)=\co(\Delta)\cup\co(\Delta^\prime)$. A \emph{consequence operator} (\emph{on~$\lang$}) is a closure operator on the partially ordered structure~$\langle\wp(\lang),\subseteq\rangle$, that is, a mapping $\co:\wp(\lang)\longrightarrow\wp(\lang)$ that satisfies, for every $\Gamma,\Delta\subseteq\lang$:
\begin{enumerate}[labelindent=\parindent, leftmargin=*,label=($\condCOT{\arabic*}$)]\setcounter{enumi}{-1}
	\item 
	\label{COpS:E}
	$\co(\Gamma)\subseteq\co(\Gamma\cup\Delta)$	

	\item 
	\label{COpS:U}
	$\Gamma\subseteq\co(\Gamma)$

	\item 
	\label{COpS:A}
	$\co(\co(\Gamma))\subseteq\co(\Gamma)$
\end{enumerate}	
Indeed, a $\TCon$-consequence relation~$\der$ on~$\lang$ induces a consequence operator~$\co_\der$ on~$\lang$ by simply setting $\co_\der(\Gamma)=\{A\in \lang:\Gamma\der A\}$, and a consequence operator~$\co$ on~$\lang$ induces a $\TCon$-consequence relation~$\der_\co$ on~$\lang$ by setting $\Pi\der_\co A$ iff $A\in\co(\Pi)$. 
Choosing to work with $\TCon$-consequence relations or with consequence operators is often just a matter of convenience, 
given that ${\der_{\co_\der}}{=}\der$ and $\co_{\der_\co}=\co$.
It is well known that the set of all $\TCon$-consequence operators on~$\lang$ equipped with the inclusion ordering constitutes a complete lattice, to which we shall refer as $\CoL_\TCon(\lang)$.

Given a consequence operator~$\co$ on~$\lang$, it is clear from \ref{COpS:U} and \ref{COpS:A} that any set of sentences of the form $\co(\Gamma)$ is a fixpoint for~$\co$.  Consider in what follows a $\TCon$-consequence relation~$\derTCon$ on~$\lang$ and a set $\Gamma\subseteq\lang$ of \emph{axioms}.  We call $\co_{\derTCon}(\Gamma)$ the \emph{$\derTCon$-theory axiomatized by~$\Gamma$}.  The elements of $\co_{\derTCon}(\Gamma)$ are dubbed its \emph{theorems}. 
We may think of a $\derTCon$-theory as the collection of assertions that any agent who accepts the generating axioms of the theory should also be committed to accept. 
A set of sentences~$\Gamma$ is called \emph{$\derTCon$-consistent} if $\co_{\derTCon}(\Gamma)\neq\lang$.  
We will use $\ty({\derTCon})=\{\co_{\derTCon}(\Gamma):\Gamma\subseteq \lang\}$ to refer to the \emph{space of all $\derTCon$-theories}.  
It is well known 
that $\ty({\derTCon})$ equipped with the inclusion ordering constitutes a complete lattice.
One of our main aims in the present paper is to propose a notion of theory that is appropriate for $\SCon$-consequence relations, taking both assertions and denials on an equal footing, and study the corresponding spaces of `$\derSCon$-theories'.

An $\SCon$-consequence relation~$\derSCon$ on $\lang$ induces an \emph{assertion-based $\TCon$-consequence relation $\derTCon_{\derSCon}$ on~$\lang$} by setting $\Pi\derTCon_{\derSCon}A$ iff $\Pi\derSCon \{A\}$, that is, $\derTCon_{\derSCon}$ consists in the restriction of $\derSCon$ to a singleton on the right-hand side. 
Clearly, such $\derTCon_{\derSCon}$ induces a corresponding space of theories, as explained above.  
As it should be expected, a $\TCon$-consequence relation might have several $\SCon$-consequence `counterparts'.  This phenomenon will be discussed in more detail in the next section.  We will see there, in particular, that a minimum such counterpart always exists, but a maximum counterpart might not exist.  
It is worth exploring the symmetry of the logical compatibility relation by observing that an $\SCon$-consequence relation~$\derSCon$ on~$\lang$ also induces a \emph{denial-based $\TCon$-consequence relation ${{}_{\derSCon}}{\derTCon}$ on~$\lang$} by setting $\Sigma\;{{}_{\derSCon}}{\derTCon}\;A$ iff $\{A\}\derSCon\Sigma$.
The above mentioned assertion-based and denial-based $\TCon$-consequence relations may be taken as representing two different `aspects' of the given~$\SCon$-consequence relation that induces them; it should be noted anyhow that the latter relation contains much more information, and it is not possible in general to recover it solely from these two specific aspects.


The next property to be considered is intended to guarantee that the compatibility of certain collections of judgments may be transferred from the (finite) local level to the (unrestrained) global level.
A logical compatibility relation~$\nder$ on~$\lang$ is called \emph{finitary} if it satisfies the following property, for every $\Pi,\Sigma\subseteq\lang$:
\begin{itemize}[leftmargin=1.5cm]
	\item[($\condCM{\fini}$)]
	if $\Pi^\prime\nder \Sigma^\prime$ for every finite $\Pi^\prime\subseteq\Pi$ and $\Sigma^\prime\subseteq\Sigma$, then $\Pi\nder\Sigma$	
\end{itemize}
%
The diverse practical incarnations of this property often turn out to be equivalent to the Axiom of Choice. 
Translated into the contexts of $\SCon$-consequence relations, $\TCon$-consequence relations, and consequence operators, respectively, finitariness may be expressed by the following statements:
\begin{itemize}[leftmargin=1.5cm]
	\item[($\condCRS{\fini}$)]
	if $\Pi\der\Sigma$, then $\Pi^\prime\der \Sigma^\prime$ for some finite $\Pi^\prime\subseteq\Pi$ and $\Sigma^\prime\subseteq\Sigma$	
\end{itemize}

\begin{itemize}[leftmargin=1.5cm]
	\item[($\condCRT{\fini}$)] 
	if $\Pi\der A$, then $\Pi^\prime\der A$ for some finite $\Pi^\prime\subseteq\Pi$	
\end{itemize}

\begin{itemize}[leftmargin=1.5cm]
	\item[($\condCOT{\fini}$)] 
	$\co(\Gamma)\subseteq\bigcup\limits_{\mbox{\scriptsize finite }\Delta\,\subseteq\,\Gamma}\co(\Delta)$	
\end{itemize}
In the concrete study of logics, finitariness is a very common property, enjoyed for instance by any logic axiomatized by finitary means or characterized by finite-valued matrices.
It is straightforward to check that the $\TCon$-con\-sequence relation $\der_\co$ is finitary if and only if the consequence operator~$\co$ is finitary.  Furthermore, for finitary consequence relations the properties \ref{comp:A}, \ref{SCon:A} and \ref{TCon:A} may clearly be simplified into the corresponding formulations in which the therein mentioned set~$\Delta$ is a singleton. 

It is not hard to check that the collection of all compatibility relations on~$\lang$ equipped with the standard inclusion ordering constitutes a complete lattice, to which we shall refer as $\ComR(\lang)$.
Obviously, the collection of all $\SCon$-consequence relations on~$\lang$ also constitutes a complete lattice under inclusion, and the same may actually be said also about the collection of all $\TCon$-consequence relations on~$\lang$.
In what follows, whenever we need to disambiguate between the two lattices of consequence relations, we will refer to the former as $\CR_\SCon(\lang)$ and refer to the latter as $\CR_\TCon(\lang)$.
Notice, in addition, that the collections of finitary consequence relations on a given set of sentences also form lattices under inclusion.

\section{Interlude on (bi)valuations}
\label{sec:semantics}

One way of actualizing the intuition that compatibility relations deal indeed with assertions and denials is by way of (bi)valuation semantics.  Let a (\emph{bi})\emph{valuation on~$\lang$} be a relation~$\biv$ on $\lang\times\{0,1\}$.  
We shall use~$\one_\biv$ to refer to the \emph{asserting aspect of~$\biv$}, namely the set $\{A\in \lang:(A,1) \in \biv\}$, and use its complement~$\zero_\biv$ to refer to the \emph{denying aspect of~$\biv$}, namely the set $\{A\in \lang:(A,0) \in \biv\}$.  
A valuation~$\biv$ on~$\lang$ is said to \emph{determine} a binary relation~$R_\biv$ on $\wp(\lang)$ defined by setting $\Pi\;R_\biv\;\Sigma$ to hold iff $\Pi\subseteq\one_\biv$ and $\Sigma\subseteq\zero_\biv$.  It should be clear that such $R_\biv$ always satisfies the property \ref{comp:E}, that $R_\biv$ satisfies the property \ref{comp:U} iff $\biv$ is functional (a.k.a.\ right-unique), and that $R_\biv$ satisfies the property \ref{comp:A} iff $\biv$ is surjective (a.k.a.\ right-total).  Accordingly, we shall call \emph{canonical valuation on~$\lang$} any total function $\biv:\lang\longrightarrow\{0,1\}$, and define a \emph{canonical semantics on~$\lang$} to be a collection of canonical valuations on~$\lang$.  
In case we are dealing with a canonical valuation~$\biv$, we will accordingly write $\Nder{\biv}$ rather than $R_\biv$.
As usual, by $\{0,1\}^\lang$ we denote the set of all canonical valuations on~$\lang$.  
The collection~$\wp(\{0,1\}^\lang)$ of all canonical semantics on~$\lang$ equipped with the standard inclusion ordering constitutes a complete lattice, to which we shall refer as $\CSem(\lang)$.
A canonical semantics~$\Sem\subseteq\{0,1\}^\lang$ is said to \emph{determine} a compatibility relation~$\Nder{\Sem} \; := \; \bigcup_{\biv\,\in\,\Sem}\Nder{\biv}$ on~$\lang$.  Obviously, the complement of such compatibility relations constitute $\SCon$-consequence relations, and so we may also say that any canonical valuation or semantics determines the associated $\SCon$-consequence and $\TCon$-consequence relations~$\derSCon_{\Sem}$ and~$\derTCon_{\Sem}$. 
It is easy to check that $\derTCon_{\derSCon_\Sem}\;=\;\derTCon_\Sem$\,, for any canonical semantics~$\Sem$, that is, the $\TCon$-consequence relation determined by~$\Sem$ coincides with the $\TCon$-consequence relation induced by the $\SCon$-consequence relation determined by~$\Sem$. 
%
In addition, given a compatibility relation~$\nder$ on~$\wp(\lang)$ and a canonical valuation~$\biv$ on~$\lang$, we shall say that~$\biv$ \emph{respects~$\nder$} if $\nder_\biv\;\subseteq\;\nder$; given a consequence relation~$\der$ on~$\wp(\lang)$ and a canonical valuation~$\biv$ on~$\lang$, we shall say that~$\biv$ \emph{respects~$\der$} if $\der\;\subseteq\;\der_\biv$.
We shall denote by $\Val{R}$ the `respectful' semantics defined by the collection of all canonical valuations on~$\lang$ that respect~$R$, where~$R$ denotes either a compatibility relation or a consequence relation.

It is worth noting that the above mappings that associate to each canonical semantics~$\Sem$ a compatibility relation~$\Nder{\Sem}$ determined by it and that associate to each compability relation~$\nder$ a respectful canonical semantics~$\Val{\nder}$
define a \emph{monotone Galois connection} between the lattice $\ComR(\lang)$ of all compatibility relations on~$\lang$ and the lattice $\CSem(\lang)$ of all canonical semantics on~$\lang$, that is, $\Nder{\Sem}\;\subseteq\;\nder$ iff $\Sem\subseteq\Val{\nder}$, for every $\nder\;\in\ComR(\lang)$ and every $\Sem\in\CSem(\lang)$.  Analogously, an \emph{antitone Galois connection} is defined between $\CR(\lang)$ and $\CSem(\lang)$, that is, $\der\;\subseteq\;\Der{\Sem}$ iff $\Sem\subseteq\Val{\der}$, for every $\der\;\in\CR(\lang)$ and every $\Sem\in\CSem(\lang)$.
From this it immediately follows that: (G1)~$\Sem\subseteq\Val{R_\Sem}$, for $R\in\{\nder,\der\}$; (G2) $\Nder{\Val{\nder}}\;\subseteq\;\nder$; (G3) $\der\;\subseteq\;\Der{\Val{\der}}$. 
Furthermore, given that any consequence relation is determined by the collection of all canonical valuations that respect it ---in other words, given that the latter collection of canonical valuations constitutes a \emph{complete} semantics for the corresponding consequence relation---, the converse of (G3) also holds good;  the converse of (G2) is seen to hold good, of course, as a corollary to that.

The property corresponding to the converse of (G1) is called \emph{absoluteness}, and it constitutes the analogue, for a logical theory, of the model-theoretic notion of `categoricity'. When it holds good, it guarantees that every compatibility / consequence relation is associated to a unique canonical semantics. 
Consider the various particular formulations for absoluteness, namely: (G1$^\nder$) $\Sem\supseteq\Val{\nder_\Sem}$; (G1$^\SCon_\der$) $\Sem\supseteq\Val{\derSCon_\Sem}$; and (G1$^\TCon_\der$) $\Sem\supseteq\Val{\derTCon_\Sem}$. 
On the one hand, it is easy to see that (G1$^\nder$) always holds good.  Indeed, suppose that the canonical valuation~$\biv$ respects $\Nder{\Sem}$, 
that is, suppose that $\Nder{\biv}\;\subseteq\;\Nder{\Sem}$. Given that obviously $\one_\biv\nder_\biv\zero_\biv$, we conclude that $\one_\biv\nder_{\biv^\star}\zero_\biv$ for some $\biv^\star\in\Sem$. This means that $\biv^\star(A)=\biv(A)$ for every $A\in\lang$, and so $\biv\in\Sem$. 
%
On the other hand, there is an essential difference in behavior between $\TCon$-consequence relations and $\SCon$-consequence relations on what concerns absoluteness, namely: (G1$^\SCon_\der$) always holds good, as a corollary to (G1$^\nder$), yet (G1$^\TCon_\der$) in general fails. 
As a matter of fact, it is not hard to identify the reason for the failure of (G1$^\TCon_\der$).  
Where $\{\biv_k\}_{k\in K}$ is a family of canonical valuations on~$\lang$, define its \emph{conjunctive combination} as the canonical valuation 
$\left[\bigcap_{k\in K}\biv_k\right]$ 
such that 
$\left[\bigcap_{k\in K}\biv_k\right](A)=1$
if $\biv_k(A)=1$ for every $k\in K$, and 
$\left[\bigcap_{k\in K}\biv_k\right](A)=0$
otherwise. 
Note first that 
$\one_{\left[\bigcap_{k\in K}\biv_k\right]}=\bigcap_{k\in K}\one_{\biv_k}$.
Next, for every $\TCon$-consequence relation $\derTCon$, note that the semantics $\Val{\derTCon}$ is closed under conjunctive combinations.  Indeed, given $\{\biv_k\}_{k\in K}\subseteq\Val{\derTCon}$, and assuming that $\Pi\derTCon A$ and that $\Pi\subseteq\one_{\left[\bigcap_{k\in K}\biv_k\right]}$, 
it follows that $\Pi\subseteq\one_{\biv_k}$ and hence $A\in\one_{\biv_k}$ for every $k\in K$.  Thus, $A\in\bigcap_{k\in K}\one_{\biv_k}=\one_{\left[\bigcap_{k\in K}\biv_k\right]}$, and we conclude that ${\left[\bigcap_{k\in K}\biv_k\right]}\in\Val{\derTCon}$. 
Nonetheless, it should be clear that not every semantics~$\Sem$ is closed under conjunctive combinations\footnote{Consider for instance a set of sentences $\lang=\{p,q\}$, and a semantics containing only two canonical valuations, $\biv_p$ and $\biv_q$, such that $\biv_x(y)=1$ iff $x=y$.  (This sort of problems related to the failure of absoluteness has been discussed as early as in~\cite{Carnap:FoL}.)}, and any witness to such phenomenon will constitute an actual counterexample to (G1$^\TCon_\der$). 
As a matter of fact, all counterexamples have this exact form: If one lets $\Sem_\cap$ denote the least superset of $\Sem$ that is closed under conjunctive combinations, one may prove that $\Val{\derTCon_\Sem}=\Sem_\cap$.  To see that, notice that every $v\in\Val{\derTCon_\Sem}$ is an intersection of canonical valuations in~$\Sem$: For each $A\notin\one_v$ one may say that $\one_v\not\der_v A$ and so there exists $v_A\in V$ such that $\one_v\not\der_{v_A} A$, that is, $\one_v\subseteq\one_{v_A}$ and $A\notin\one_{v_A}$; clearly, $v=\bigcap_{A\,\notin\,\one_v}v_A$.

The absoluteness of compatibility relations and of $\SCon$-consequence relations guarantees that $\ComR(\lang)$ and $\CR_\SCon(\lang)$ inherit from $\CSem(\lang)$ the structure of a complete Boolean algebra (under the obvious inclusion ordering).  In contrast, the complete lattice $\CR_\TCon(\lang)$ fails absoluteness and in fact fails to be distributive, in general. Failing absoluteness, a $\TCon$-consequence relation may happen to be determined by two distinct semantics.  All that we can guarantee is that $\derTCon_{\Sem}{=}\derTCon_{\Sem^\prime}$ if and only if $\Sem_\cap=\Sem^\prime_\cap$.  
However, if $\Sem_\cap=\Sem^\prime_\cap$ yet $\Sem\neq\Sem^\prime$, then we still have $\derSCon_{\Sem}{\neq}\derSCon_{\Sem^\prime}$ in spite of $\derTCon_{\Sem}{=}\derTCon_{\Sem^\prime}$.
This means that an 
assertion-based %
$\TCon$-consequence relation~$\derTCon$ may boast distinct $\SCon$-consequence relations as its `counterparts', namely, there will in general exist%
\footnote{For a straightforward class of examples (cf.\ \cite{JM:ineffable}), let $\biv_{\bbtop}$ denote the the `dadaistic' valuation on~$\lang$ such that $\one_{\biv_{\bbtop}}=\lang$, consider a semantics~$\Sem$ such that $\biv_{\bbtop}\not\in\Sem$, and let $\Sem^\star:=\Sem\cup\{\biv_{\bbtop}\}$.  Then $\derTCon_{\Sem}\;=\;\derTCon_{\Sem^\star}$, and $\lang\derSCon_{\Sem}\varnothing$, while $\lang\nder_{\Sem^\star}\varnothing$.} 
relations $\derSCon_1\;\neq\;\derSCon_2$ such that $\derTCon_{\derSCon_1}\;=\;\derTCon_{\derSCon_2}$ 
(an analogous observation may be formulated, of course, concerning denial-based $\TCon$-consequence relations and their multiple possible `counterparts' in terms of $\SCon$-consequence).
That phenomenon suggests that one should associate to $\SCon$-consequence relations a space of theories that has a richer structure than the space of theories of the $\TCon$-consequence relations induced by the former relations.  
Studying the spaces of `$\derSCon$-theories' is indeed one of our main goals in the present paper.

Delving a bit deeper, we may note that a $\TCon$-consequence relation~$\derTCon$ always has a \emph{minimum} (\emph{generalized}) \emph{counterpart $\derMini{\derTCon}$} characterized by setting $(\Pi,\Sigma)\in\derMini{\derTCon}$ iff $\Pi\derTCon A$ for some $A\in\Sigma$. It is easy to see that $\derMini{\derTCon}$ corresponds to the largest possible set of canonical valuations that respect~$\derTCon$, that is, $\derMini{\derTCon}\;{=}\derSCon_{\Val{\derTCon}}$. To that effect it suffices to check that $\Pi\derSCon_{\Sem_\cap}\Sigma$ if and only if $\Pi\derTCon_\Sem A$ for some $A\in\Sigma$.
In contrast, a maximum generalized counterpart for a $\TCon$-consequence relation~$\der$ may not exist. Indeed, in general, there may not exist a minimal set of canonical valuations~$\Sem$ such that $\Sem_\cap={\val}(\derTCon)$. 
Such a minimal set of canonical valuations, and the corresponding \emph{maximum} (\emph{generalized}) \emph{counterpart $\derMaxi{\derTCon}$} for a given $\TCon$-consequence relation~$\derTCon$ may be shown to exist, in particular, whenever~$\derTCon$ happens to be finitary.  Indeed, in this case one may use the Lindenbaum-Asser lemma%
: When it exists, the minimal set of canonical valuations corresponds precisely to the so-called \emph{relatively maximal} theories of~$\derTCon$, that is, it consists in the set of all canonical valuations~$\biv$ such that $\one_\biv\in\ty(\derTCon)$ and such that $A\in\co_{\derTCon}(\Gamma)$ for some sentence $A\notin\one_\biv$ and for every~$\Gamma$ such that $\one_\biv\subsetneq \Gamma$. 
Given a canonical semantics $\Sem\subseteq \{0,1\}^\lang$, 
and using~$\co_\Sem$ to denote $\co_{\derTCon_\Sem}$, 
it is easy to see that:
\begin{equation*}
\co_{\Sem}(\Gamma)
=\bigcap\limits_{\biv\,\in\, \Sem\mbox{\scriptsize\rm\ such that }{\one_\biv\,\supseteq\, \Gamma}}\one_\biv.
\tag{$\star$} \label{eq:TspaceTh}
\end{equation*}

For a simple illustration involving $\SCon$-consequence, consider first the $\TCon$-consequence relation $\derTCon_{CPL}$ of \emph{classical propositional logic}, determined by the set of all Boolean valuations.  On the one hand, its maximum counterpart 
consists precisely in the $\SCon$-consequence relation $\derSCon_{CPL}$ determined by the set of all Boolean valuations; on the other hand, its minimum counterpart,
determined by the closure for conjunctive combinations of the set of all Boolean valuations, seems to be of much lesser interest.  
In particular, when we consider the disjunction connective, $\lor$, we see that the consecution $(\{A \lor B\},\{A,B\})$ belongs to the maximum counterpart of~$\derTCon_{CPL}$, intuitively saying thus that one cannot simultaneously deny both sentences~$A$ and~$B$ without also denying the sentence $A \lor B$, but it clearly does not belong to the minimum counterpart of~$\derTCon_{CPL}$.
On what concerns maximum counterparts, considering now in more detail the implication connective, $\imp$, let~$\At$ be a denumerable set of variables and let~$\lang$ be the least set with $P\subseteq\lang$ and such that if $A,B\in\lang$ then $A\to B\in\lang$.  The maximum counterpart 
in this language 
of the classical $\TCon$-consequence relation
is~%
the $\SCon$-consequence relation~$\vartriangleright_\Sem$
where 
$\biv(A\to B)=0$ if and only if $\biv(A)=1$ and $\biv(B)=0$, for every $A,B\in\lang$ %
 and every $\biv\in \Sem$%
; it is indeed not difficult to check that it actually consists in the least consequence relation~$\derSCon$ such that:
$$\{A,A\imp B\}\derSCon \{B\}\qquad\textrm{and}\qquad \{B\}\derSCon \{A\imp B\}\qquad\textrm{and}\qquad \varnothing\derSCon \{A,A\imp B\}.$$ 
While the first of the latter three statements, which concerns the consecution $(\{A,A\imp B\}, \{B\})$, intuitively says that one cannot simultaneously assert~$A$ and deny~$B$ while also asserting~$A\imp B$, the second statement says~that one cannot assert~$B$ while denying $A\imp B$, and the third one says that one cannot deny~$A$ while also denying~$A\imp B$.  So, according to what we wrote at the beginning of the present section, when assertion is represented by the truth-value~$1$ and denial is represented by the truth-value~$0$, the three statements above obviously describe, as expected, the well-known truth-function for classical implication.

For an interesting non-classical illustration, is worth noting that if one considers instead the consequence relation $\derTCon_{IPL}$ of \textit{intuitionistic propositional logic}, it is known that its maximum counterpart is determined by the bivaluations defined through the usual \textit{Kripke} (\textit{topological}) \textit{semantics}, while its minimum counterpart is determined by the valuations defined through the so-called \textit{Beth interpretation} (check \cite{Gabbay:SIiHIL}, ch.2 \&~3).

Our final case study in this section aims at reinforcing the importance of looking for an adequate notion of consequence operator and an adequate notion of theory to deal with $\SCon$-consequence relations; it involves a simple non-fini\-tary $\SCon$-consequence relation that induces a finitary 
assertion-based
$\TCon$-consequence relation.
Take $\lang=\nat\cup\{\exists\}$, let~$\biv_\exists$ be the canonical valuation such that $\one_{\biv_\exists}=\{\exists\}$, and consider $\derSCon_\Sem$ determined by $\Sem=\{0,1\}^\lang\setminus\{\biv_\exists\}$. It is not difficult to check that~$\derSCon_\Sem$ is the least consequence relation~$\der$ such that 
$\{\exists\}\der\nat$; 
note that such~$\vartriangleright$ is not finitary, because $\{\exists\}\not\der\Psi$ for each finite $\Psi\subseteq\nat$ (one may check that the latter consecution holds good by considering the canonical valuation $\biv_\Psi\in \Sem$ such that $\one_{\biv_\Psi}=\{\exists\}\cup(\nat\setminus\Psi)$).
However, the consequence relation $\derTCon_{\derSCon_\Sem}$ is finitary (and actually constitutes the minimal $\TCon$-consequence relation definable on~$\lang$); indeed, $\Gamma\derTCon_{\derSCon_\Sem} A$ if and only if $A\in\Gamma$, for every $\Gamma\cup\{A\}\subseteq\lang$ (and this is confirmed by noticing that $\biv_\exists\in \Sem_\cap$, namely because $\biv_\exists=\bigcap\limits_{n\,\in\,\nat}v_{\{n\}}$).
Do note, by the way, that the consecution $\{\exists\}\der\nat$ has a clear `denial-based reading', namely: 
$\biv(\exists)=0$ whenever $\biv(\nat)=\{0\}$.

A well-known result about consequence relations (cf.\ Theorem 1.3.5 of \cite{Wojcicki:TLC})
states that a $\TCon$-consequence relation~$\der$ is finitary precisely when $\ty(\der)$ is closed under ultraproducts.  We will explore this topic in more detail later, in Section~\ref{sec:theories}, on what concerns $\SCon$-consequence and the associated notion of theory-pair that we will introduce there.  One may already observe, nonetheless, that the last illustration above shows that the mentioned result does not carry over to $\SCon$-consequence relations by considering only the space $\ty(\derTCon_{\derSCon})$ of induced $\derTCon$-theories, 
given that the latter space is closed under ultraproducts, in spite of the fact that the associated $\SCon$-consequence relation~$\derSCon$, is non-finitary.  
With its bias towards assertion, we see that $\TCon$-consequence relations can only furnish thus a partial ---though sometimes convenient--- view of the more generous logical phenomenon captured by the more symmetrical notion of $\SCon$-consequence, which allows for a more balanced take on assertion and denial.

\section{Generalizing consequence operators}
\label{sec:genConsOp}

In this section we will take action concerning the shortcomings of $\TCon$-consequence pointed out in Section~\ref{sec:semantics} and investigate a definition of consequence operator that properly fits the more symmetric notion of consequence relation given by $\SCon$-consequence which arises as a natural dual to the notion of logical compatibility explored in Section~\ref{sec:consequence}.  
We want to be able thus to account for a notion of logical theory that is simultaneously based both on a set of primitively asserted sentences and on a set of primitively denied sentences. 
Recall from Section~\ref{sec:consequence} that the received notion of consequence operator, satisfying properties \ref{COpS:E}, \ref{COpS:U} and \ref{COpS:A}, was closely associated to $\TCon$-consequence: in fact, the lattices $\CR_\TCon(\lang)$ and $\CoL_\TCon(\lang)$ are dually isomorphic.
On our quest to find a notion of consequence operator naturally associated to $\SCon$-consequence, it will be convenient from now on to distinguish among two types of consequence operators, the former one that we will henceforth 
rechristen `$\TCon$-consequence operator', and a novel one that will be referred to below as~$\coS$ and will be dubbed an `$\SCon$-consequence operator'.

Consider in what follows an $\SCon$-consequence relation~$\derSCon$ on~$\lang$ and let $\Gamma_1,\Gamma_0\subseteq\lang$ be distinguished sets of sentences of~$\lang$, respectively taken as \emph{axioms} and \emph{anti-axioms} of a certain logical theory.
Fixed in the background any set~$\Sigma$ of sentences taken by assumption to be denied, let $\coD{\derSCon}{\Sigma}(\Gamma_1)$ denote the set 
containing every sentence~$A\in\lang$ such that $\Gamma_1\derSCon\Sigma\cup\{A\}$;
fixed in the background any set~$\Pi$ of sentences taken by assumption to be asserted, let $\coA{\derSCon}{\Pi}(\Gamma_0)$ denote the set 
containing every sentence $A\in\lang$ such that $\{A\}\cup\Pi\derSCon\Gamma_0$.
%
The elements of $\coD{\derSCon}{\Sigma}(\Gamma_1)$ might be thought of
as the sentences that one is committed to assert once the sentences in~$\Gamma_1$ are all asserted, in the context of the denial of all the sentences in~$\Sigma$; 
or informally as the theorems of the theory axiomatized by~$\Gamma_1$ modulo the denied sentences of~$\Sigma$.
Analogously, the elements of $\coA{\derSCon}{\Pi}(\Gamma_0)$ might be thought of
as the sentences that one is committed to deny once the sentences in~$\Gamma_0$ are all denied, in the context of the assertion of all the sentences in~$\Sigma$; 
or informally as the anti-theorems of the theory anti-axiomatized by~$\Gamma_0$ modulo the asserted sentences of~$\Pi$.
%
When the background context of judgment contains neither sentences to be taken by assumption as asserted nor sentences to be taken by assumption as denied, it should be clear, in particular, that $\coD{\derSCon}{\varnothing}$ and $\coA{\derSCon}{\varnothing}$ describe, respectively, the assertion-based and the denial-based $\TCon$-consequence operators associated to~$\derSCon$, introduced in Section~\ref{sec:consequence}.

From now on, 
for every $\Delta\subseteq\lang$, we will use~$\overline{\Delta}$ to refer to $\lang\setminus\Delta$.  
Moreover, 
we will write $(\Gamma_1,\Gamma_0)\subseteq (\Gamma_1',\Gamma_0')$ instead of both $\Gamma_1\subseteq\Gamma_1'$ and $\Gamma_0\subseteq\Gamma_0'$, and given a family $\{(\Pi_i,\Sigma_i)\}_{i\in I}\subseteq\wp(\lang)\times\wp(\lang)$ of consecutions, we will write $\bigcap_{i\in I}(\Pi_i,\Sigma_i)$ instead of $(\bigcap_{i\in I}\Pi_i,\bigcap_{i\in I}\Sigma_i)$.
The following result collects some fundamental properties of the operators associated to a given $\SCon$-consequence relation according to the above definitions:

\begin{prop}\label{prop:sfcrs}
Let $\derSCon$ be an $\SCon$-consequence relation on~$\lang$, and let $\Pi,\Sigma\subseteq \lang$. Then, we have:
\begin{itemize}
\item[{\rm (1)}] $\coD{\derSCon}{\Sigma}$ 
and $\coA{\derSCon}{\Pi}$ are $\TCon$-consequence operators~on $\lang\,$.

\item[{\rm (2)}] If $\Sigma\subseteq\Sigma'$ then $\coD{\derSCon}{\Sigma}\subseteq\coD{\derSCon}{\Sigma'}$, and if $\Pi\subseteq\Pi'$ then $\coA{\derSCon}{\Pi}\subseteq\coA{\derSCon}{\Pi'}$.

\item[{\rm (3)}] For each $\Omega\subseteq \lang$, either both $\coD{\derSCon}{\overline{\Omega}}(\Omega)=\Omega$ and $\coA{\derSCon}{\Omega}(\overline{\Omega})=\overline{\Omega}\,$, \\or else $\coD{\derSCon}{\overline{\Omega}}(\Omega)=\coA{\derSCon}{\Omega}(\overline{\Omega})=\lang\,$.

\item[{\rm (4)}] $\coD{\derSCon}{\Sigma}(\Pi)=\bigcap\limits_{\Omega\,\subseteq\,\lang{\scriptsize\mbox{\rm\ such that }}(\Omega,\overline{\Omega})\,\supseteq\,(\Pi,\Sigma)}\coD{\derSCon}{\overline{\Omega}}(\Omega)$\ \ and\\ $\coA{\derSCon}{\Pi}(\Sigma)=\bigcap\limits_{\Omega\,\subseteq\,\lang{\scriptsize\mbox{\rm\ such that }}(\Omega,\overline{\Omega})\,\supseteq\,(\Pi,\Sigma)}\coA{\derSCon}{\Omega}(\overline{\Omega})\,$.

\item[{\rm (5)}] 
If $\coA{\derSCon}{\Pi}(\Sigma)=\Sigma'$, then $\coD{\derSCon}{\Sigma}=\coD{\derSCon}{\Sigma'}$, 
and 
if $\coD{\derSCon}{\Sigma}(\Pi)=\Pi'$, then $\coA{\derSCon}{\Pi}=\coA{\derSCon}{\Pi'}$. 

\item[{\rm (6)}] 
If $\ndisjS{\coD{\derSCon}{\Sigma}(\Pi)}{\coA{\derSCon}{\Pi}(\Sigma)}$, then $\coD{\derSCon}{\Sigma}(\Pi)=\coA{\derSCon}{\Pi}(\Sigma)=\lang\,$.
\end{itemize}
\end{prop}
\begin{proof}
We prove each of the listed properties, in turn:
\begin{itemize}
\item[(1)] We show that $\coD{\derSCon}{\Sigma}$ is a $\TCon$-consequence operator. The proof for $\coA{\derSCon}{\Pi}$ is analogous.

\ref{COpS:E}
If $\Gamma_1\subseteq\Gamma_1'$ and $\Gamma_1\derSCon \Sigma\cup\{A\}$, then $\Gamma_1'\derSCon \Sigma\cup\{A\}$ follows from \ref{SCon:E}, 
thus $\coD{\derSCon}{\Sigma}(\Gamma_1)\subseteq\coD{\derSCon}{\Sigma}(\Gamma_1')$.

\ref{COpS:U}
If $A\in\Gamma_1$ then $\Gamma_1\derSCon \Sigma\cup\{A\}$ follows from \ref{SCon:U},
for any $\Sigma\subseteq\lang$,
and thus $A\in\coD{\derSCon}{\Sigma}(\Gamma_1)$.

\ref{COpS:A}
Let $T_1:=\coD{\derSCon}{\Sigma}(\Gamma_1)$, and suppose $A\in\coD{\derSCon}{\Sigma}(T_1)$.  We want to show that $A\in T_1$.  
By assumption, we have that $T_1\derSCon\Sigma\cup\{A\}$.  
By definition, it is also the case that $\Gamma_1\derSCon\Sigma\cup\{B\}$ for each $B\in T_1$.
Take an arbitrary $\Omega\subseteq\lang$.  If either $\ndisjS{\Gamma_1}{\overline{\Omega}}$ or $\ndisjS{\Sigma}{\Omega}$ or $A\in\Omega$, then $\Omega\cup\Gamma_1\derSCon\Sigma\cup\{A\}\cup\overline{\Omega}$ follows from \ref{SCon:A}.
If, on the contrary, $\disjS{\Gamma_1}{\overline{\Omega}}$ and $\disjS{\Sigma}{\Omega}$ and $A\not\in\Omega$, then either 
$T_1\subseteq\Omega$
and so $\Omega\cup\Gamma_1\derSCon\Sigma\cup\{A\}\cup\overline{\Omega}$ follows by \ref{SCon:E} from $T_1\derSCon\Sigma\cup\{A\}$, or there exists a $B\in T_1$ such that $B\in\overline{\Omega}$ 
and so $\Omega\cup\Gamma_1\derSCon\Sigma\cup\{A\}\cup\overline{\Omega}$ follows by \ref{SCon:E} from $\Gamma_1\derSCon\Sigma\cup\{B\}$.
Thus, using \ref{SCon:A} with $\Delta=\lang$, we may conclude that $\Gamma_1\derSCon\Sigma\cup\{A\}$, and so $A\in \coD{\derSCon}{\Sigma}(\Gamma_1)=T_1$.

\item[(2)] If $\Sigma\subseteq\Sigma'$ and $\Gamma_1\derSCon A,\Sigma$, then $\Gamma_1\der A,\Sigma'$ follows from 
\ref{SCon:E}.
The proof of the other assertion is analogous.

\item[(3)] Let $\Omega\subseteq \lang$. 
If $\Omega\derSCon\overline{\Omega}$ then $\Omega\derSCon \overline{\Omega}\cup\{A\}$ and $\{A\}\cup\Omega\derSCon\overline{\Omega}$ for every $A\in \lang$, using \ref{SCon:E}, so $\coD{\derSCon}{\overline{\Omega}}(\Omega)=\coA{\derSCon}{\Omega}(\overline{\Omega})=\lang$.
Otherwise, let $\Omega\nder\overline{\Omega}$, where~$\nder$ denotes the complement of~$\derSCon$.  This would immediately imply, in particular, that $\Omega\nder\overline{\Omega}\cup\{A\}$ for every $A\in\overline{\Omega}$.  This means that $\disjS{\overline{\Omega}}{\coD{\derSCon}{\overline{\Omega}}(\Omega)}$.  
But we also know that $\Omega\subseteq\coD{\derSCon}{\overline{\Omega}}(\Omega)$, in view of \ref{COpS:U}, proved in item~(1) above.  It thus follows that $\coD{\derSCon}{\overline{\Omega}}(\Omega)=\Omega$.
A similar argument may be used to show that in such a situation we also have $\coA{\derSCon}{\Omega}(\overline{\Omega})=\overline{\Omega}$.

\item[(4)] We prove the first assertion; the second one is proved analogously. On the one hand, 
it should be clear that $\coD{\derSCon}{\Sigma}(\Pi)\subseteq\bigcap\limits_{\Omega\,\subseteq\,\lang{\scriptsize\mbox{\rm\ such that }}(\Omega,\overline{\Omega})\,\supseteq\,(\Pi,\Sigma)}\coD{\derSCon}{\overline{\Omega}}(\Omega)$.  
Indeed, for any $\Omega\subseteq\lang$ such that $(\Pi,\Sigma)\subseteq(\Omega,\overline{\Omega})$ we conclude from items (1) and (2) above that $\coD{\derSCon}{\Sigma}(\Pi)\subseteq\coD{\derSCon}{\overline{\Omega}}(\Omega)$.
For the converse inclusion, suppose that $(\Pi,\Sigma)\subseteq(\Omega,\overline{\Omega})$ implies
$A\in\coD{\derSCon}{\overline{\Omega}}(\Omega)$, and consider an arbitrary $\Omega\subseteq\lang$.
So, in case we do have $(\Pi,\Sigma)\subseteq(\Omega,\overline{\Omega})$, it obviously follows that $\Omega\cup\Pi\derSCon\Sigma\cup\{A\}\cup\overline{\Omega}$.  Otherwise, it must the case that $\ndisjS{\Pi}{\overline{\Omega}}$ or $\ndisjS{\Sigma}{\Omega}$, and in either situation \ref{SCon:U} gives us $\Omega\cup\Pi\derSCon\Sigma\cup\{A\}\cup\overline{\Omega}$.  We may now invoke \ref{SCon:A} to conclude that $\Pi\derSCon\Sigma\cup\{A\}$, that is, $A\in\coD{\derSCon}{\Sigma}(\Pi)$.

\item[(5)] 
Again, we check the first assertion in detail; the  second one is analogous.
Half of the proof is straightforward: 
Given that $\Sigma\subseteq\coA{\derSCon}{\Pi}(\Sigma)=\Sigma'$, in view of item (1) above, we conclude from item (2) above that $\coD{\derSCon}{\Sigma}(\Pi)\subseteq\coD{\derSCon}{\Sigma'}$.
For the converse inclusion, in view of item (4) above, we need to show that if 
[a]~$A\in\coD{\derSCon}{\overline{\Omega}}(\Omega)$ whenever 
$(\Pi,\coA{\derSCon}{\Pi}(\Sigma))\subseteq(\Omega,\overline{\Omega})$,
then 
[b]~$A\in\coD{\derSCon}{\overline{\Omega}}(\Omega)$ whenever 
$(\Pi,\Sigma)\subseteq(\Omega,\overline{\Omega})$.
So, assume [a] and let $\Omega\subseteq\lang$ be such that $(\Pi,\Sigma)\subseteq(\Omega,\overline{\Omega})$.
From the item (3) above we know that either [c]~$\coA{\derSCon}{\Omega}(\overline{\Omega})=\overline{\Omega}$ or [d]~$\coA{\derSCon}{\Omega}(\overline{\Omega})=\lang$.  
In case [c], given that $(\Pi,\Sigma)\subseteq(\Omega,\overline{\Omega})$ we know from items (1) and (2) that $\coA{\derSCon}{\Pi}(\Sigma)\subseteq\coA{\derSCon}{\Omega}(\overline{\Omega})=\overline{\Omega}$, and from [a] we conclude that $A\in\coD{\derSCon}{\overline{\Omega}}(\Omega)$.  In case [d], it is clear that $A\in\coD{\derSCon}{\overline{\Omega}}(\Omega)=\lang$.

\item[(6)] 
%
Let $\Pi':=\coD{\derSCon}{\Sigma}(\Pi)$ and $\Sigma':=\coA{\derSCon}{\Pi}(\Sigma)$.  We are assuming that $\ndisjS{\Pi'}{\Sigma'}$.
Using items (1), (5) and (4) above we may infer that $\Pi'=\coD{\derSCon}{\Sigma}(\Pi)=\coD{\derSCon}{\Sigma}(\coD{\derSCon}{\Sigma}(\Pi))=\coD{\derSCon}{\Sigma'}(\Pi')=\bigcap\limits_{\Omega\,\subseteq\,\lang{\scriptsize\mbox{\rm\ such that }}(\Omega,\overline{\Omega})\,\supseteq\,(\Pi',\Sigma')}\coD{\derSCon}{\overline{\Omega}}(\Omega)$.  Given the assumption that $\ndisjS{\Pi'}{\Sigma'}$, we see that in the latter chain of equalities we are talking about the intersection of an empty family of elements of~$\wp(\lang)$, from what we conclude that $\Pi'=\lang$.
The proof that $\Sigma'=\lang$ is analogous.
\qedhere
\end{itemize}
\end{proof}

The above result shows that there is in fact a plethora of $\TCon$-consequence operators that could be associated to a single given $\SCon$-consequence relation, and such operators may be collectively organized into a rich structure, being in particular monotonic with respect to the underlying contextual background of assumptions constituted by certain primitively asserted / denied sentences.  In reading the above properties, it is worth noting that the mentioned operators are completely determined by $\coD{\derSCon}{\overline{\Omega}}(\Omega)$ and $\coA{\derSCon}{\Omega}(\overline{\Omega})$, for $\Omega\subseteq\lang$.  The properties also suggest that the sets of axioms and anti-axioms that originate consistent `generalized theories' manage to separate in between theorems and anti-theorems, as further discussed in Section~\ref{sec:theories}.  Before building upon these properties, though, in proposing a full generalization of $\TCon$-consequence operators, we shall also analyze the situation concerning finitariness.

\begin{prop}\label{prop:finit}
Let $\derSCon$ be an $\SCon$-consequence relation on~$\lang$. Then, $\derSCon$ is finitary if and only if the following two properties hold for every $\Pi,\Sigma\subseteq \lang${\rm :}
\begin{itemize}
\item[\rm {[A]}] $\coD{\derSCon}{\Sigma}$ and $\coA{\derSCon}{\Pi}$ are finitary
\item[\rm {[B]}] $\coD{\derSCon}{\Sigma}=\bigcup\limits_{\mbox{\rm\scriptsize finite }\Sigma^\star\subseteq\,\Sigma}\coD{\derSCon}{\Sigma^\star}$\ and\ \ $\coA{\derSCon}{\Pi}=\bigcup\limits_{\mbox{\rm\scriptsize finite }\Pi^\star\subseteq\,\Pi}\coA{\derSCon}{\Pi^\star}$
\end{itemize}
\end{prop}
\begin{proof}
First, assume that~$\derSCon$ is finitary. 
We prove that {[A$^\prime$]} $\coD{\derSCon}{\Sigma}$ is finitary and {[B$^\prime$]} $\coD{\derSCon}{\Sigma}=\bigcup\limits_{\mbox{\rm\scriptsize finite }\Sigma^\star\subseteq\,\Sigma}\coD{\derSCon}{\Sigma^\star}$; the proofs of the other halves of {[A]} and {[B]} are analogous. 
Take $\Gamma_1\cup\{A\}\subseteq\lang$ and suppose $A\in\coD{\derSCon}{\Sigma}(\Gamma_1)$, that is, $\Gamma_1\derSCon \Sigma\cup\{A\}$. As~$\derSCon$ is finitary, there must exist finite sets $\Gamma_1^\star\subseteq\Gamma_1$ and $\Delta^\star\subseteq\Delta$ such that $\Gamma_1^\star\derSCon \Delta^\star\cup\{A\}$.  This means that $A\in\coD{\derSCon}{\Sigma^\star}(\Gamma_1^\star)$.  
But we also know from Prop.~\ref{prop:sfcrs}(1)~\&~(2) that
$\coD{\derSCon}{\Sigma^\star}(\Gamma_1^\star)\subseteq\coD{\derSCon}{\Sigma}(\Gamma_1^\star)$ and 
$\coD{\derSCon}{\Sigma^\star}(\Gamma_1^\star)\subseteq\coD{\derSCon}{\Sigma^\star}(\Gamma_1)$, and we reach thereby the envisaged conclusions.

Assume now that {[A]} and {[B]} both hold good, and take $\Pi,\Sigma\subseteq \lang$ such that $\Pi\derSCon\Sigma$. 
If $\Pi\cup\Sigma$ is finite we are done. Otherwise, either~$\Pi$ or~$\Sigma$ is infinite, and hence non-empty. In case $\Pi\neq\varnothing$, let $A\in\Pi$. Clearly, we have $\{A\}\cup\Pi\derSCon\Sigma$ and so $A\in\coA{\derSCon}{\Pi}(\Sigma)$. 
Using {[A]} and {[B]}, we see that there must exist finite sets $\Pi^\star\subseteq\Pi$ and $\Sigma^\star\subseteq\Sigma$ such that $A\in\coA{\derSCon}{\Pi^\star}(\Sigma^\star)$. Therefore, we have $\{A\}\cup\Pi^\star\derSCon\Sigma^\star$. The case where $\Sigma\neq\varnothing$ is checked analogously.
\end{proof}

It should be clear, in view of the latter result, that finitariness of $\SCon$-consequence must be witnessed both from the viewpoint of axioms~/ anti-axioms and from the viewpoint of background asserted~/ denied assumptions.  This observation will be made particularly relevant later on, in Prop.~\ref{prop:fingco}.

Taking advantage of the previous results, we may say that each $\SCon$-consequence relation~$\der\;\subseteq\wp(\lang)\times\wp(\lang)$ induces a certain operator $\coS_\der:\wp(\lang)\times\wp(\lang)\longrightarrow\wp(\lang)\times\wp(\lang)$ by setting $\coS_\der\left(\Pi,\Sigma\right):=\left(\coD{\der}{\Sigma}\left(\Pi\right),\coA{\der}{\Pi}\left(\Sigma\right)\right)$.
The following properties are now immediate to check:

\begin{prop}\label{prop:sscr}
Let $\der$ be an $\SCon$-consequence relation on~$\lang$, and let $\Pi,\Sigma\subseteq\lang$. Then, we have:
\begin{itemize}
\item[\rm(1)] $\coS_{\der}$ is a closure operator on the partially ordered structure $\langle\wp(\lang)\times\wp(\lang),\subseteq\rangle$.
\item[\rm(2)]  $\coS_{\der}(\Pi,\Sigma)=\bigcap\limits_{\Omega\,\subseteq\,\lang\mbox{\rm\scriptsize\ such that }(\Omega,\overline{\Omega})\,\supseteq\,(\Pi,\Sigma)}\coS_{\der}(\Omega,\overline{\Omega})$.
\item[\rm(3)] Either\ \ $\coS_{\der}(\Omega,\overline{\Omega})=(\Omega,\overline{\Omega})$\ \ or\ \ 
$\coS_{\der}(\Omega,\overline{\Omega})=(\lang,\lang)$.
\end{itemize}
\end{prop}
\begin{proof}
Concerning item~{(1)}, note that properties \ref{COpS:E} and \ref{COpS:U} follow from Prop.~\ref{prop:sfcrs}{(1)} 
and property \ref{COpS:A} follows from Prop.~\ref{prop:sfcrs}{(5)}.
Item~{(2)} follows from Prop.~\ref{prop:sfcrs}{(4)}, 
and 
item {(3)} follows from Prop.~\ref{prop:sfcrs}{(3)}.
\end{proof}

The above properties suggest the following abstract `bilateralist' definition, simultaneously generalizing the purely assertion-based and the purely denial-based $\TCon$-consequence operators associated to a given $\SCon$-consequence relation.  An \emph{$\SCon$-consequence operator on~$\lang$} is defined as a mapping $\coS:\wp(\lang)\times\wp(\lang)\longrightarrow\wp(\lang)\times\wp(\lang)$ that satisfies, for every $\Gamma_0,\Gamma_1,\Delta_0,\Delta_1\subseteq\lang$:
\begin{enumerate}[labelindent=\parindent, leftmargin=*,label=($\condCOS{\arabic*}$)]\setcounter{enumi}{-1}
	\item 
	\label{C2OpS:E}
	$\coS(\Gamma_1,\Gamma_0)\subseteq\coS(\Gamma_1\cup\Delta_1,\Gamma_0\cup\Delta_0)$	

	\item 
	\label{C2OpS:U}
	$(\Gamma_1,\Gamma_0)\subseteq\co2(\Gamma_1,\Gamma_0)$

	\item 
	\label{C2OpS:A}
	$\coS(\coS(\Gamma_1,\Gamma_0))\subseteq\coS(\Gamma_1,\Gamma_0)$

	\item 
	\label{C2OpS:T}
	$\coS(\Gamma_1,\Gamma_0)\supseteq\bigcap\limits_{\Omega\,\subseteq\,\lang\mbox{\rm\scriptsize\ such that }(\Omega,\overline{\Omega})\,\supseteq\,(\Gamma_1,\Gamma_2)}\coS_{\der}(\Omega,\overline{\Omega})$
\end{enumerate}	

Properties \ref{C2OpS:E}, \ref{C2OpS:U} and \ref{C2OpS:A} are precisely those required to make $\coS$ a closure operator on $\langle\wp(\lang)\times\wp(\lang),\subseteq\rangle$.  
They are clearly not sufficient, though, to capture the intimate dependencies between the thereby involved assertions and denials.  As it so happens, such dependencies are fully captured by property \ref{C2OpS:T}.  The latter property is so strong, however, that a simpler (though somewhat less familiar-looking) characterization of $\SCon$-consequence operators is now available:

\begin{prop}\label{prop:gco}
The mapping $\coS:\wp(\lang)\times\wp(\lang)\longrightarrow\wp(\lang)\times\wp(\lang)$ is an $\SCon$-consequence operator on $\lang$ if and only if the following properties both hold:
\begin{itemize}
	\item[\rm (V)] either\ \ $\coS(\Omega,\overline{\Omega})=(\Omega,\overline{\Omega})$\ \ or\ \ 
$\coS(\Omega,\overline{\Omega})=(\lang,\lang)$, for every $\Omega\subseteq\lang$
	\item[\rm (T)] 
		$\coS(\Gamma_1,\Gamma_0)=\bigcap\limits_{
	\substack{\Omega\,\subseteq\,\lang\mbox{\scriptsize\rm\ such that }\\(\Omega,\overline{\Omega})\,\supseteq\,(\Gamma_1,\Gamma_0)\mbox{\scriptsize\rm\ and }\coS(\Omega,\overline{\Omega})\,\neq\,(\lang,\lang)}}(\Omega,\overline{\Omega})$, for every $\Gamma_0,\Gamma_1\subseteq\lang$
\end{itemize}
\end{prop}
\begin{proof}
Let $\coS$ be an $\SCon$-consequence operator. In order to establish property (V), let $(\Gamma_1,\Gamma_0):=\coS(\Omega,\overline{\Omega})$ and assume $(\Gamma_1,\Gamma_0)\neq(\Omega,\overline{\Omega})$, for some $\Omega\subseteq\lang$. By \ref{C2OpS:U}, we conclude that $(\Omega,\overline{\Omega})\subsetneq (\Gamma_1,\Gamma_0)$. Thus, either there exists some $A\in\Gamma_1\setminus\Omega$, and so $A\in \overline{\Omega}\subseteq\Gamma_0$, or else there exists some $A\in\Gamma_0\setminus\overline{\Omega}$, and so $A\in {\Omega}\subseteq\Gamma_1$.  In both cases we see that $\ndisjS{\Gamma_1}{\Gamma_0}$. Hence, from \ref{C2OpS:T} we know that $(\Gamma_1,\Gamma_0)\supseteq(\lang,\lang)$, for $(\Gamma_1,\Gamma_0)$ contains the intersection of an empty family of elements of~$\wp(\lang)\times\wp(\lang)$, and from this it follows that $\Gamma_0=\Gamma_1=\lang$.
Now, concerning property (T), note in particular that \ref{C2OpS:T} guarantees its right-to-left inclusion, in view of (V) and the fact that $(\Gamma_1,\Gamma_0)\cap(\lang,\lang)=(\Gamma_1,\Gamma_0)$.  For the converse inclusion, note that if there is some $\Omega\subseteq\lang$ such that $(\Gamma_1,\Gamma_0)\subseteq(\Omega,\overline{\Omega})$ then \ref{C2OpS:E} guarantees that $\coS(\Gamma_1,\Gamma_0)\subseteq\coS(\Omega,\overline{\Omega})$. 
Further knowing that $\coS(\Omega,\overline{\Omega})\neq(\lang,\lang)$, we may now use (V) to conclude that $\coS(\Omega,\overline{\Omega})=(\Omega,\overline{\Omega})$. Therefore, $\coS(\Gamma_1,\Gamma_0)\subseteq(\Omega,\overline{\Omega})$, and property (T) follows.

Conversely, assume now that $\coS$ satisfies (V) and (T). 
Since $(\Gamma_1,\Gamma_0)\subseteq(\Gamma_1\cup\Gamma_1',\Gamma_0\cup\Gamma_0')$ it is clear that $(\Gamma_1\cup\Gamma_1',\Gamma_0\cup\Gamma_0')\subseteq(\Omega,\overline{\Omega})$ implies $(\Gamma_1,\Gamma_0)\subseteq(\Omega,\overline{\Omega})$. 
Property (T) then guarantees that $\coS(\Gamma_1,\Gamma_0)\subseteq\coS(\Gamma_1\cup\Gamma_1',\Gamma_0\cup\Gamma_0')$, and property \ref{C2OpS:E} thus holds. 
Note that \ref{C2OpS:U} follows easily from (T) since the intersection is over pairs $(\Omega,\overline{\Omega})$ with $(\Gamma_1,\Gamma_0)\subseteq(\Omega,\overline{\Omega})$. 
Concerning property \ref{C2OpS:A}, and using (T) again, it suffices to note that if $(\Gamma,\Delta)\subseteq(\Omega,\overline{\Omega})$ then \ref{C2OpS:E} and (V) tell us that $\coS(\Gamma_1,\Gamma_0)\subseteq\coS(\Omega,\overline{\Omega})=(\Omega,\overline{\Omega})$.
Finally, property \ref{C2OpS:T} follows directly from~(T), given that $(\Gamma_1,\Gamma_0)\cap(\lang,\lang)=(\Gamma_1,\Gamma_0)$ and that property (V) states that if $\coS(\Omega,\overline{\Omega})\neq(\lang,\lang)$ then $\coS(\Omega,\overline{\Omega})=(\Omega,\overline{\Omega})$.
\end{proof}

It is worth noticing that \ref{C2OpS:A} had no role to play in proving the first part of the latter result.  Indeed, we have just checked in the course of the full proof of that result that \ref{C2OpS:A} follows from \ref{C2OpS:E}, \ref{C2OpS:U} and \ref{C2OpS:T}; as a matter of fact, we have only included it in the definition of $\SCon$-consequence operators for the sake of producing a more familiar-looking definition --- one that would be more easily comparable to the earlier, and standard, definition of $\TCon$-consequence operators.

As one would surely expect, an $\SCon$-consequence operator~$\coS$ will be called \emph{finitary} when it satisfies the following property, for every $\Gamma_0,\Gamma_1\subseteq\lang$:
\begin{itemize}[leftmargin=1.5cm]
	\item[(\condCOS{\fini})] $\coS(\Gamma_1,\Gamma_0)\subseteq
	\bigcup\limits_{\substack{\mbox{\scriptsize finite }\Gamma_1^\star\subseteq\,\Gamma_1\\ \mbox{\scriptsize finite }\Gamma_0^\star\subseteq\,\Gamma_0}}\coS(\Gamma_1^\star,\Gamma_0^\star)$.
\end{itemize}

\begin{prop}\label{prop:fingco}
An $\SCon$-consequence relation $\der$ is finitary if and only if $\coS_\der$ is finitary.
\end{prop}
\begin{proof}
The result is immediate from the definition of the $\SCon$-consequence operator $\coS_\der$ induced by~$\der$, and by property (\condCOS{\fini}), if we observe that properties [A] and [B] of Prop.~\ref{prop:finit} are jointly equivalent to requiring for all $\Gamma_1,\Gamma_0\subseteq\lang$ that both
$\coD{\der}{\Gamma_0}(\Gamma_1)\subseteq\bigcup\limits_{\substack{\mbox{\scriptsize finite }\Gamma_1^\star\subseteq\,\Gamma_1\\ \mbox{\scriptsize finite }\Gamma_0^\star\subseteq\,\Gamma_0}}\coD{\der}{\Gamma_0^\star}(\Gamma_1^\star)$ and $\coA{\der}{\Gamma_1}(\Gamma_0)\subseteq\bigcup\limits_{\substack{\mbox{\scriptsize finite }\Gamma_1^\star\subseteq\,\Gamma_1\\ \mbox{\scriptsize finite }\Gamma_0^\star\subseteq\,\Gamma_0}}\coA{\der}{\Gamma_1^\star}(\Gamma_0^\star)$.
\end{proof}

We shall end the present section by showing that the newly introduced $\SCon$-con\-se\-quence operators are closely associated to $\SCon$-consequence relations, in precisely the same way as it was known to happen with $\TCon$-consequence. 
To that effect, we will say that an $\SCon$-consequence operator~$\coS$ induces a certain binary relation~$\der_{\coS}$ on~$\wp(\lang)$ by setting $\Pi\der_{\coS}\Sigma$ iff $\coS(\Pi,\Sigma)=(\lang,\lang)$.  

\begin{prop}\label{prop:ssok}
Let $\coS$ be an $\SCon$-consequence operator, and $\der$ be an $\SCon$-consequence relation, both on $\lang$.
We have that: 
\begin{itemize}
\item[\rm (1)] $\der_{\coS}$ is an $\SCon$-consequence relation on $\lang$.
\item[\rm (2)] ${\der_{{\coS}_\der}}{=}\der$ and $\coS_{\der_{\coS}}=\coS$.
\item[\rm (3)] $\der_{\coS}$ is finitary if and only if $\coS$ is finitary.
\end{itemize}
\end{prop}
\begin{proof}
We prove each of the above statemens, in turn:
\begin{itemize}
\item[\rm (1)] Concerning \ref{SCon:E}, given $\Pi\der_{\coS}\Sigma$ we have $\coS(\Pi,\Sigma)=(\lang,\lang)$ and from that we conclude using \ref{C2OpS:E} that $\coS(\Pi\cup\Pi',\Sigma\cup\Sigma')=(\lang,\lang)$, and so $\Pi'\cup\Pi\der_{\coS}\Sigma\cup\Sigma'$.  As for property \ref{SCon:U}, if $\ndisjS{\Pi}{\Sigma}$, then by invoking (T) with the intersection of an empty family of elements of~$\wp(\lang)$ we conclude that $\coS(\Pi,\Sigma)=(\lang,\lang)$, and so $\Pi\der_{\coS}\Sigma$.  Finally, concerning \ref{SCon:A}, suppose that $\Omega\cup\Pi\der_{\coS}\Sigma\cup\overline{\Omega}$, that is, $\coS(\Pi\cup\Omega,\Sigma\cup\overline{\Omega})=(\lang,\lang)$, for every $\Omega\subseteq\lang$.  In that case, given $(\Pi,\Sigma)\subseteq(\lang,\lang)$ the said hypothesis amounts more simply to $\coS(\Omega,\overline{\Omega})=(\lang,\lang)$. But then (T) immediately gives us $\coS(\Pi,\Sigma)=(\lang,\lang)$, that is, $\Pi\der_{\coS}\Sigma$.
\item[\rm (2)] We prove first that ${\der_{{\coS}_\der}}{=}\der$.  Note that $\Pi\der_{{\coS}_\der}\Sigma$ iff $\coS_\der(\Pi,\Sigma)=(\lang,\lang)$, by definition of the $\SCon$-con\-sequence relation induced by the $\SCon$-consequence operator~$\coS_\der$, and recall that $\coS_\der\left(\Pi,\Sigma\right)=\left(\coD{\der}{\Sigma}(\Pi),\coA{\der}{\Pi}(\Sigma)\right)$, by definition of the $\SCon$-consequence operator induced by the $\SCon$-consequence relation~$\der$.  Thus, $\Pi\der_{{\coS}_\der}\Sigma$ iff $\coD{\der}{\Sigma}(\Pi)=\lang=\coA{\der}{\Pi}(\Sigma)$ iff [a] both $\Pi\der\Sigma\cup\{A\}$ and $\{A\}\cup\Pi\der\Sigma$ for every $A\in\lang$.  On the one hand, taking $\Delta=\{A\}$ in \ref{SCon:A}, we may conclude from [a] that $\Pi\der\Sigma$.  On the other hand, from $\Pi\der\Sigma$ we may obtain [a] by using \ref{SCon:E}.

We show next that $\coS_{\der_{\coS}}=\coS$.  Unravelling the definitions, 
we may say that 
$\coS_{\der_{\coS}}(\Pi,\Sigma)=(T_1,T_0)$, where $T_1:=\{A\in\lang:\coS(\Pi,\Sigma\cup\{A\})=(\lang,\lang)\}$ and $T_0:=\{A\in\lang:\coS(\Pi\cup\{A\},\Sigma)=(\lang,\lang)\}$.
In view of (T) and (V), it would suffice to prove that $\coS(\Pi,\Sigma\cup\{A\})=(\lang,\lang)$ if and only if $A\in\Omega$ follows from assuming both $(\Pi,\Sigma)\subseteq(\Omega,\overline{\Omega})$ and $\coS(\Omega,\overline{\Omega})\neq(\lang,\lang)$, for any given~$\Omega\subseteq\lang$, as well as to prove an analogous result concerning $\coS(\Pi\cup\{A\},\Sigma)=(\lang,\lang)$ and $A\in\overline{\Omega}$.  We shall check in detail just the first stated equivalence.
Assume first that [b] $(\Pi,\Sigma)\subseteq(\Omega,\overline{\Omega})$ and [c] $\coS(\Omega,\overline{\Omega})\neq(\lang,\lang)$, for an arbitrary~$\Omega\subseteq\lang$, and suppose that [d] $\coS(\Pi,\Sigma\cup\{A\})=(\lang,\lang)$. From [d] and [b], it follows from (T) that $A\not\in\overline{\Omega}$, on the pain of contradicting [c], and so we conclude that $A\in\Omega$.
Next, suppose that $\coS(\Pi,\Sigma\cup\{A\})\neq(\lang,\lang)$.  From (T) we may then conclude that [e] $\coS(\Omega,\overline{\Omega})\neq(\lang,\lang)$ for some $\Omega\subseteq\lang$ such that [f] $(\Pi,\Sigma\cup\{A\})\subseteq(\Omega,\overline{\Omega})$.  But from [f] it follows that both [g] $(\Pi,\Sigma)\subseteq(\Omega,\overline{\Omega})$ and [h] $A\in\overline{\Omega}$ are the case.  Note that [h] means that $A\not\in\Omega$, in spite of both [e] and [g] being the case.
\item[\rm (3)] This follows from Prop.~\ref{prop:fingco}, using the fact that $\coS_{\der_{\coS}}=\coS$, proved in item (2).
\qedhere
\end{itemize}
\end{proof}

The above result makes it evident that, in a precise sense, choosing to work with $\SCon$-consequence relations or with $\SCon$-consequence operators is just a matter of convenience ---
we have just seen that there is a dual isomorphism between the lattice $\CR_\SCon(\lang)$ of $\SCon$-consequence relations on~$\lang$ and the lattice $\CoL_\SCon(\lang)$ defined by the set of all $\SCon$-consequence operators on~$\lang$ equipped with the inclusion ordering, generalizing the already mentioned corresponding well-known result concerning $\TCon$-consequence; we have also seen that the result carries over to the finitary case. %
Incidentally, insofar as convenience is involved, in the next section we will see that the newly introduced notion of $\SCon$-consequence operator allows one in fact to entertain a generalized notion of `theory' that manages to treat denials on a par with assertions.

\section{Generalizing the spaces of theories}
\label{sec:theories}

As mentioned in Section~\ref{sec:consequence}, the standard space of all theories of a given $\TCon$-consequence relation~$\derTCon$ has the structure of a complete lattice under inclusion.  It allows one to study the connections between different sets of axioms, or hypotheses, added on top of a given logic, and the corresponding collections of theorems thereby generated, or whereby concluded.  This is fully adequate, we admit, as long as one is interested, say, in seeing how the acceptance of certain judgments commit an agent to accepting other judgments (or, dually, in seeing how the rejection of certain judgments leads an agent to rejecting other judgments).  In the present paper, though, we are interested in analyzing the behavior of theories containing \textit{both} assertions and denials.  We want to impose no preference on acceptance over rejection, nor vice-versa, and want to consider instead theories that allow an agent to take either attitude with respect to given judgments, that is, theories that allow an agent to assert some sentences while simultaneously denying other sentences.  As we shall argue and illustrate, such a setting demonstrably conveys a greater expressive power for defining and reasoning about theories.


Given an $\SCon$-consequence relation~$\der$ on~$\lang$, a set $\Gamma_1\subseteq\lang$ of \emph{axioms} and a set $\Gamma_0\subseteq\lang$ of \emph{anti-axioms}, 
we call $\coS_{\der}(\Gamma_1,\Gamma_0)$ the \emph{$\der$-theory-pair axiomatized by $(\Gamma_1,\Gamma_0)$}. 
Where $(T_1,T_0):=\coS_{\der}(\Gamma_1,\Gamma_0)$, we call the elements of~$T_1$ the \emph{theorems} and call the elements of~$T_0$ the \emph{anti-theorems} of the given $\der$-theory-pair. 
We say that the pair $(\Gamma_1,\Gamma_0)$ is \emph{$\der$-consistent} if $\coS_{\der}(\Gamma_1,\Gamma_0)\neq (\lang,\lang)$, and \emph{$\der$-inconsistent} otherwise.

\begin{lem}\label{lemma:consistency}
Let $\der$ be an $\SCon$-consequence relation on~$\lang$ and 
let $(T_1,T_0)$ be the $\der$-theory-pair axiomatized by $(\Gamma_1,\Gamma_0)\in\wp(\lang)\times\wp(\lang)$.
The following properties are equivalent:
\begin{itemize}
\item[{\rm[a]}] $(\Gamma_1,\Gamma_0)$ is $\der$-inconsistent
\item[{\rm[b]}] $\ndisjS{T_1}{T_0}$
\item[{\rm[c]}] $\Gamma_1\der\Gamma_0$
\end{itemize}
\end{lem}
\begin{proof}
{[a]} implies {[b]}: If $(\Gamma_1,\Gamma_0)$ is $\der$-inconsistent then $T_1=T_0=\lang=T_0\cap T_1$.  To round off, just recall that we always assume~$\lang$ to be non-empty.
{[b]} implies {[c]}: If $A\in T_0\cap T_1$ then we have both $\Gamma_1\der \Gamma_0\cup\{A\}$ and $\{A\}\cup \Gamma_1\der \Gamma_0$. Given an arbitrary $\Omega\subseteq \lang$, note that $(\Gamma_1,\Gamma_0)\not\subseteq(\Omega,\overline{\Omega})$ implies $\Omega\cup\Gamma_1\der\Gamma_0\cup\overline{\Omega}$, in view of \ref{SCon:U}, while $(\Gamma_1,\Gamma_0)\subseteq(\Omega,\overline{\Omega})$ implies $\Omega\cup\Gamma_1\der\Gamma_0\cup\overline{\Omega}$ by case analysis over $A\in\Omega$ and $A\not\in\Omega$, in view of \ref{SCon:E}.
So, by using \ref{SCon:A} we conclude that $\Gamma_1\der \Gamma_0$.
{[c]} implies {[a]}: If $\Gamma_1\der\Gamma_0$ then \ref{SCon:E} guarantees that $\Gamma_1\der \Gamma_0\cup\{A\}$ and $\{A\}\cup\Gamma_1\der \Gamma_0$ for each $A\in\lang$. Thus $T_1=T_0=\lang$ and $(\Gamma_1,\Gamma_0)$ is $\der$-inconsistent.
\end{proof}


For a given $\SCon$-consequence relation~$\der$ we will use $\ty^\SCon(\der):=\{\coS_{\der}(\Gamma_1,\Gamma_0):\Gamma_0,\Gamma_1\subseteq L\}$ to refer to the \emph{space of all $\der$-theory-pairs}. 
As $\coS_{\der}$ is a closure operator, it follows that $\ty^\SCon(\der)$ equipped with the component-wise inclusion ordering constitutes a complete lattice, appropriately generalizing thus the situation for $\TCon$-consequence relations and thereby associated $\TCon$-consequence operators.
%
%
The phenomena illustrated in Section~\ref{sec:semantics} according to which $\TCon$-consequence relations in general fail absoluteness and also may in general be associated to several distinct $\SCon$-consequence relations have indicated that the traditional study of $\TCon$-consequence falls short in providing a full understanding of $\SCon$-consequence.
In contrast, it is clear that $\der$-theory-pairs provide a full-fledged generalization of the notion of theory for $\TCon$-consequence.
Indeed, given an $\SCon$-consequence relation $\derSCon$ on~$\lang$, the first component of the pair $\coS_{\derSCon}(\Gamma_1,\varnothing)$ amounts precisely to the assertion-based theory $\co_{{\derTCon}_{\!\derSCon}}(\Gamma_1)$, while the second component contains all anti-theorems that follow according to~$\derSCon$ in the context of taking~$\Gamma_1$ as background asserted assumptions; the situation is entirely dual with respect to $\coS_{\derSCon}(\varnothing,\Gamma_0)$, the denial-based theory $\co_{{{}_{\derSCon}}{\derTCon}}(\Gamma_0)$ that constitutes its second component and its corresponding first component consisting of anti-theorems that follow according to~$\derSCon$ in the context of taking~$\Gamma_0$ as background denied assumptions.
Note that $\der$-theory-pairs can go much beyond these particular cases, however, as they allow one to deal with theorems and anti-theorems that result from simultaneously taking both sets of axioms and anti-axioms as non-empty. 

It is worth looking at theory-pairs in the light of the semantic results of Section~\ref{sec:semantics}. 
Given a canonical semantics~$\Sem\subseteq\{0,1\}^\lang$, let's use $\coS_\Sem$ to denote $\coS_{\der_{\Sem}}$. For any $\Omega\subseteq\lang$, we have that $\coS_{\Sem}(\Omega,\overline{\Omega})\neq(\lang,\lang)$ if and only if $\Omega\not\der_{\Sem}\overline{\Omega}$ if and only if $(\Omega,\overline{\Omega})=(\one_\biv,\zero_\biv)$ for some $\biv\in \Sem$. Consequently, also generalizing the traditional setting, we have:
\begin{equation*}
\coS_{\Sem}(\Gamma_1,\Gamma_0)=\bigcap\limits_{\biv\,\in\,\Sem\mbox{\scriptsize\rm\ such that }(\one_\biv,\zero_\biv)\,\supseteq\, (\Gamma_1,\Gamma_0)}(\one_\biv,\zero_\biv).
\tag{$\star\star$} \label{eq:SspaceTh}
\end{equation*}
It should be clear that the pairs $(\one_\biv,\zero_\biv)$ with $\biv\in \Sem$ are 
not only fixed points of~$\coS$, but they actually cons\-titute precisely
the \emph{maximal} theory-pairs of $\ty^\SCon(\der_\Sem)$, in the sense that they are $\der_\Sem$-consistent but their only proper extension is the inconsistent pair $(\lang,\lang)$.
Note also that here the absoluteness of $\SCon$-consequence has the effect of making $\ty^\SCon(\der_\Sem)$ unique, for here each valuation has a `countermodelling' role to play.

Recall that the standard notion of \textit{theory} as a `closed set of sentences' is deeply connected to a closure operator~$\der$ in the following sense: the fixed points of the closure operator~$\co_\der$ identify precisely the sets of sentences that contain all of its $\TCon$-consequences, that is, if there is some $\Gamma\in\wp(\lang)$ such that $T=\co_\der(\Gamma)$, then $T\der A$ iff $A\in T$, in view of \ref{TCon:A} and \ref{TCon:U}.  This is now easily seen to generalize to our novel bilateralist notion of closure:

\begin{prop}\label{prop:closure}
Let $\derSCon$ be an $\SCon$-consequence relation.  If there is some $(\Gamma_1,\Gamma_0)\in\wp(\lang)\times\wp(\lang)$ such that $(T_1,T_0)=\coS_{\derSCon}(\Gamma_1,\Gamma_0)$, then: {\rm[a]} $T_1\der T_0\cup\{A\}$ iff $A\in T_1$ and {\rm [b]} $\{A\}\cup T_1\der T_0$ iff $A\in T_0$.
\end{prop}
\begin{proof}
Recall from Section~\ref{sec:genConsOp} that $\coS_{\derSCon}\left(\Gamma_1,\Gamma_0\right):=\left(\coD{\derSCon}{\Gamma_0}\left(\Gamma_1\right),\coA{\derSCon}{\Gamma_1}\left(\Gamma_0\right)\right)$. 
Thus, given that $T_1$ is taken here as the set $\coD{\derSCon}{\Gamma_0}\left(\Gamma_1\right)$, we may invoke Prop.~\ref{prop:sfcrs}(1) to conclude that $T_1=\coD{\derSCon}{T_0}\left(T_1\right)$, which means precisely that $A\in T_1$ if and only if $T_1\der T_0\cup\{A\}$, as stated in item [a].  The proof of item [b] is analogous.
\end{proof}

\noindent
The latter result fully supports the intuition that led us to upgrade the view of a theory as a closed set of a `unilateralist' closure operator, which has a bias exclusively towards assertion (or, dually, exclusively towards denial), into a more generous framework in which we look for \emph{closed set-pairs} (pairs of sentences that satisfy items [a] and [b] in the statement of the above proposition) to fit our generalized notion of closure, which takes both assertions and denials equally into consideration.
It is worth indeed pointing out that, as in the case of $\TCon$-consequence, the bilateralist $\SCon$-consequence operators are uniquely determined by their associated spaces of theories.

We shall now briefly revisit the first specific illustration from Section~\ref{sec:semantics}, that of the $\SCon$-consequence relation for the im\-pli\-cation-only fragment of classical propositional logic, to show how the present notion of theory-pair provides a framework for the study of consequence that is strictly more expressive than the received one-sided standard notion of theory.  
Let $\derSCon$ be the mentioned $\SCon$-consequence relation and $\derTCon_{\derSCon}$ be the thereby induced $\TCon$-consequence relation.
Let $q\in P$ be a propositional variable, and consider the $\der$-theory-pair $(T_1,T_0):=\coS_{\derSCon}(\varnothing,\{q\})$.
Using the above characterization of $\coS_\Sem$, it is clear that such $\der$-theory-pair is semantically characterized by taking~$\Sem$ as the set of all Boolean valuations~$\biv$ such that $\biv(q)=0$. 
Hence, it follows in particular that $T_1=\co_{\derTCon}(\{q\to A:A\in\lang\})$. 
The present theoretical framework for $\SCon$-consequence has, in this case, two clear advantages.
First, note that it still allows one to obtain~$T_0$ as the set of all sentences evaluated to~$0$ by all Boolean valuations~$\biv\in\Sem$ such that $\biv(q)=0$. 
Second, we claim that~$T_1$ simply cannot be finitely axiomatized as a~$\derTCon$-theory (recall that negation is not expressible in the implication-only fragment of classical propositional logic), that is, for every finite $\Psi\subseteq\lang$ we have $\co_{\derTCon}(\Psi)\neq T_1$. To see this, given a finite $\Psi\subseteq\lang$ take any variable $r\neq q$ in~$P$ such that~$r$ does not occur in the sentences in~$\Psi$; let~$\biv$ be the canonical bivaluation such that $\biv(r)=0$ and $\biv(p)=1$ if $p\neq r$; then, it is immediate to see that $\Psi\subseteq\one_v$ yet $\biv(q\to r)=0$.

We will close this section with a result that further reinforces the adequacy of our present novel notion of theory-pair for $\SCon$-consequence relations.  As we have seen in the final illustration of Section~\ref{sec:semantics}, it may happen that the $\TCon$-consequence relation $\derTCon_{\derSCon}$ is finitary even when the $\SCon$-consequence relation that induces it is not finitary.  It is well-known 
(cf.\ Theorem 1.3.5 of \cite{Wojcicki:TLC})
however, that a $\TCon$-consequence relation is finitary precisely when its space of theories is `closed under ultraproducts'.  This suggests that $\TCon$-consequence fails somehow in capturing all the nuances of $\SCon$-consequence.  Nonetheless, we can now show that the mentioned result concerning closure under ultrafilters may be generalized in a very natural way with the help of $\derSCon$-theory-pairs.

For immediate subsequent use, we briefly recall a couple of standard definitions and facts 
(cf., e.g.,~\cite{ChangKeisler}).
Given a set~$I$, an \emph{ultrafilter on $I$} is a set $\mathcal{U}\subseteq\wp(I)$ satisfying the following conditions: {(U1)} ~$\varnothing\notin\mathcal{U}$;  {(U2)}~if $X\in\mathcal{U}$ and $X\subseteq Y$ then $Y\in\mathcal{U}$; {(U3)}~if $X,Y\in\mathcal{U}$ then $X\cap Y\in \mathcal{U}$; {(U4)}~if $X\subseteq I$, then either $X\in \mathcal{U}$ or $I \setminus X\in \mathcal{U}$. Further, $\mathcal{W}\subseteq\wp(I)$ is said to enjoy the \emph{finite intersection property} if every finite subset $\{X_1,\dots,X_n\}\subseteq\mathcal{W}$ is such that $X_1\cap\dots\cap X_n\neq\varnothing$. 
It is useful to bear in mind that the Ultrafilter Lemma states that for every $\mathcal{W}\subseteq\wp(I)$ with the finite intersection property there exists an ultrafilter $\mathcal{U}$ such that $\mathcal{W}\subseteq\mathcal{U}$. 
Finally, the \emph{ultraproduct $\mbox{\Cross}_{\mathcal{U}}(\mathcal{T})$ of a family $\mathcal{T}:=\{(\Pi_i,\Sigma_i)\}_{i\in I}\subseteq\wp(\lang)\times\wp(\lang)$ modulo an ultrafilter $\mathcal{U}$ on~$I$} is the pair $(\Pi_{\mathcal{U}},\Sigma_{\mathcal{U}})$ with
$\Pi_{\mathcal{U}}=\{A\in\lang:\{i\in I:A\in\Pi_i\}\in\mathcal{U}\}\textrm{ and }\Sigma_{\mathcal{U}}=\{A\in\lang:\{i\in I:A\in\Sigma_i\}\in\mathcal{U}\}$.

\begin{lem}\label{prop:ultracons}
Let $\der$ be an $\SCon$-consequence relation on~$\lang$, let $\mathcal{T}:=\{(\Pi_i,\Sigma_i)\}_{i\in I}\subseteq\wp(\lang)\times\wp(\lang)$, and let~$\mathcal{U}$ be an ultrafilter on~$I$. If $(\Pi_i,\Sigma_i)$ is $\der$-consistent for each $i\in I$ then $\mbox{\Cross}_{\mathcal{U}}(\mathcal{T})$ is $\der$-consistent.
\end{lem}
\begin{proof}
Suppose that $\mbox{\Cross}_{\mathcal{U}}(\mathcal{T})$ is $\der$-inconsistent, that is, $(\Gamma_{\mathcal{U}},\Delta_{\mathcal{U}})=(\lang,\lang)$. Take $A\in\lang$. This means that $\{i\in I:A\in\Pi_i\}\in\mathcal{U}$ and $\{i\in I:A\in\Sigma_i\}\in\mathcal{U}$, and therefore also $\{i\in I:A\in\Pi_i\mbox{ and }A\in\Sigma_i\}=\{i\in I:A\in\Pi_i\}\cap\{i\in I:A\in\Sigma_i\}\in\mathcal{U}$ and hence $\{i\in I:A\in\Pi_i\mbox{ and }A\in\Sigma_i\}\neq\varnothing$. Thus, $\ndisjS{\Pi_i}{\Sigma_i}$ for some $i\in I$ and, by Lemma~\ref{lemma:consistency}, $(\Pi_i,\Sigma_i)$ is $\der$-inconsistent.
\end{proof}

\begin{prop}\label{prop:ultra}
Let $\der$ be an $\SCon$-consequence relation on~$\lang$. Then, $\der$ is finitary if and only if $\ty^\SCon(\der)$ is closed under ultraproducts.
\end{prop}
\begin{proof}
Assume that $\der$ is a finitary $\SCon$-consequence relation. Let $\mathcal{U}$ be an ultrafilter on~$I$, and consider a family $\mathcal{T}:=\{(X_i,Y_i)\}_{i\in I}\subseteq\ty^\SCon(\der)$, that is, for each $i\in I$, assume that there exist $\Pi_i,\Sigma_i\subseteq\lang$ such that $X_i=\coD{\der}{\Sigma_i}(\Pi_i)$ and $Y_i=\coA{\der}{\Pi_i}(\Sigma_i)$. 
To show that $(\Pi_{\mathcal{U}},\Sigma_{\mathcal{U}}):=\mbox{\Cross}_{\mathcal{U}}(\mathcal{T})$ is in~$\ty^\SCon(\der)$, we have to find $\Pi,\Sigma\subseteq\lang$ such that $\Pi_{\mathcal{U}}=\coD{\der}{\Sigma}(\Pi)$ and $\Sigma_{\mathcal{U}}=\coA{\der}{\Pi}(\Sigma)$.
The obvious candidate for playing the role of such pair~$(\Pi,\Sigma)$ is $(\Pi_{\mathcal{U}},\Sigma_{\mathcal{U}})$ itself, and we claim indeed that the ultraproduct $(\Pi_{\mathcal{U}},\Sigma_{\mathcal{U}})$ is a closed set-pair of~$\der$. %
Let $A\in\lang$ and suppose first that $\Pi_{\mathcal{U}}\der \Sigma_{\mathcal{U}}\cup\{A\}$. 
As $\der$ is finitary, there exist finite sets $\Pi^\star:=\{B_1,\ldots,B_n\}\subseteq\Pi_{\mathcal{U}}$ and $\Sigma^\star:=\{C_1,\ldots,C_m\}\subseteq\Sigma_{\mathcal{U}}$ such that $\Pi^\star\der \Sigma^\star\cup\{A\}$, for some $n,m\in\nat$. 
Set $I_{B_k}:=\{i\in I:B_k\in\Pi_i\}$, 
for $1\leq k\leq n$, 
and set $I_{C_k}:=\{i\in I:C_k\in\Sigma_i\}$, 
for $1\leq k\leq m$. 
By the definitions of~$\Pi_{\mathcal{U}}$ and~$\Sigma_{\mathcal{U}}$ it is clear that $I_{B_k}\in\mathcal{U}$, for $1\leq k\leq n$, and $I_{C_k}\in\mathcal{U}$, for $1\leq k\leq m$. 
Setting $I_D:=I_{B_1}\cap\dots\cap I_{B_n}\cap I_{C_1}\cap\dots\cap I_{C_m}$, by (U3) we may conclude that $I_D\in\mathcal{U}$. 
Let $i\in I_D$.  This amounts to assuming $\Pi^\star\subseteq X_i$ and $\Sigma^\star\subseteq Y_i$.  So, in view of $\Pi^\star\der \Sigma^\star\cup\{A\}$ we may use \ref{SCon:E} to conclude that $X_i\der Y_i\cup\{A\}$.
%
We may then use Prop.~\ref{prop:closure}[a] to conclude that $A\in X_i$.
It follows that $I_D\subseteq\{i\in I:A\in X_i\}$, which implies by (U2) that $\{i\in I:A\in X_i\}\in\mathcal{U}$, and so we see that $A\in\Pi_{\mathcal U}$.
The proof that $\{A\}\cup\Pi_{\mathcal{U}}\der \Sigma_{\mathcal{U}}$ implies $A\in \Sigma_{\mathcal{U}}$ is analogous.

Conversely, assume now that $\ty^\SCon(\der)$ is closed under ultraproducts, consider arbitrary $\Pi,\Sigma\subseteq\lang$, and suppose $\Pi^\star\not\der\Sigma^\star$ for all finite $\Pi^\star\subseteq\Pi$ and $\Sigma^\star\subseteq\Sigma$. Set $I:=\{(\Pi^\star,\Sigma^\star): \textrm{finite }\Pi^\star\subseteq\Pi,\textrm{finite }\Sigma^\star\subseteq\Sigma\}$ and set $\mathcal{W}:=\{\bullet i: i\in I\}$, 
where $\bullet i:=\{J\subseteq I:i\in J\}$. 
As a finite union of finite sets is still finite, it is straightforward to see that~$\mathcal{W}$ enjoys the finite intersection property, and thus, by the Ultrafilter Lemma, there exists an ultrafilter~$\mathcal{U}$ on~$I$ that contains~$\mathcal{W}$. 
Consider now the family $\mathcal{T}:=\{(\Pi_i,\Sigma_i)\}_{i\in I}$, where $(\Pi_i,\Sigma_i):=\coS_\der(i)$ for each $i \in I$.  Note that each such $(\Pi_i,\Sigma_i)$ is a $\der$-consistent theory-pair, by hypothesis.
By assumption, the ultraproduct $(\Pi_{\mathcal{U}},\Sigma_{\mathcal{U}}):=\mbox{\Cross}_{\mathcal{U}}(\mathcal{T})$ is in~$\ty^\SCon(\der)$, and is $\der$-consistent as a consequence of Lemma~\ref{prop:ultracons}. 
Hence, Lemma~\ref{lemma:consistency} implies that $\Pi_{\mathcal{U}}\not\der\Sigma_{\mathcal{U}}$. 
We claim that $(\Pi,\Sigma)\subseteq(\Pi_{\mathcal{U}},\Sigma_{\mathcal{U}})$. %
Indeed, if $A\in\Pi$ then $\bullet(\{A\},\varnothing)\in\mathcal{U}$, 
so $\bullet(\{A\},\varnothing)\subseteq\{i\in I:A\in\Pi_i\}$ 
and by (U2) it follows that $\{i\in I:A\in\Pi_i\}\in\mathcal{U}$, from which we conclude that $A\in \Pi_{\mathcal{U}}$; the proof that $\Sigma\subseteq\Sigma_{\mathcal{U}}$ is analogous.
Using \ref{comp:E}, we conclude then from $(\Pi,\Sigma)\subseteq(\Pi_{\mathcal{U}},\Sigma_{\mathcal{U}})$ that $\Pi\not\der\Sigma$. 
\end{proof}

\section*{What should follow?}
\label{conclusion}

David Hilbert once defended the axiomatic method as essential in providing a ``definitive presentation and complete logical assurance of the content of our knowledge'' (cf.~\cite{hilbert:zahlbegriff}).
Nowadays, axiomatically presented theories are commonplace, both in Mathematics and in Science.  
Logics, in particular, have often been presented through consequence relations induced by sets of axioms.  For that purpose, nonetheless, sets of anti-axioms could equally be used, though perhaps the force of habit has made such alternative approach much less prevalent.  While it is true that classical logic and intuitionistic logic disagree on what should count as theorems, they agree on what should count as anti-theorems (for an anti-axiomatization of classical propositional logic, see \cite{Morgan}).  As long as one delves further into non-classical territory, however, situations in which anti-theorems play a more important role than theorems start to look just as natural.
The novel notion of a theory-pair explored in the present paper is flexible enough to allow for the study of logics from either perspective, and even from both perspectives at once.
For a further natural step, from the model-theoretic viewpoint, it would seem worth pursuing more investigations in the line of~\cite{Badia:Marcos:18}, concerning the characterization ``in purely mathematical terms'' ---in the sense of~\cite{Tarski:notions}--- of classes of structures presented not only in terms of collections of equations or of equality-free positive literals but also classes of structures whose presentations include negative literals, which happen to be accommodated equally well within our present approach to consequence.

From the proof-theoretic viewpoint, it is usual to consider sets of sentences endowed with an algebraic structure and to demand the underlying notion of consequence to be `substitution-invariant', that is, to assume that $\Pi\der\Sigma$ implies $\sigma(\Pi)\der\sigma(\Sigma)$ for any endomorphism~$\sigma$ on~$\lang$.
Semantically, in order to properly cope with substitution-invariance, one often considers richer interpretation structures such as `logical matrices', as in~\cite{CzelakowskiJ} and~\cite{Zygmunt}, rather than bivaluations. 
Such an enrichment would not really bring added complexity to the hereby proposed notion of space of theories for $\SCon$-consequence relations. 
It is worth pointing out, at any rate, that the consequence operators $\coD{\derSCon}{\Sigma}$ or $\coA{\derSCon}{\Pi}$ are not necessarily substitution-invariant even when~$\derSCon$ is, in view of the fixed background assertions or denials.
Deductively, the richer structure of $\SCon$-consequence relations is no longer compatible with the modest design that is characteristic of Hilbert-style calculi. Several alternatives are at hand, though, including of course Gentzen-like sequent calculi, but a simpler adequate possibility is just to consider 
a generalization of Hilbert-style
calculi whose rules contain sets of premises and sets of conclusions, as in~\cite{ShoesmithSmiley} and~\cite{smar:ccal:GenHilbert}.

As we have pointed out above, exactly like the space of theories associated to a given $\TCon$-consequence relation, the space of theory-pairs associated to a given $\SCon$-consequence relation also forms a complete lattice --- as expected, meets are given by (component-wise) intersections of theory-pairs, and joins correspond to the closure of (component-wise) unions of theory-pairs.  
We believe it would be interesting to investigate the distributivity of these lattices of theories, namely by looking for sufficient conditions such as the existence of appropriate disjunction connectives in the case of $\TCon$-consequence (cf.~\cite{MartinPollard}).  We expect such a study to further enlighten the dissimilarities between $\SCon$-consequence and $\TCon$-consequence, well patent in the fact that the lattice of all $\SCon$-consequence relations is always distributive, whereas that is not the case for the lattice of all $\TCon$-consequence relations (cf.~\S 1.17 of~\cite{Humberstone:Connectives}).
In the same spirit, many other aspects of the spaces of theory-pairs would seem worth investigating.  We have for instance shown that finitary $\SCon$-consequence relations correspond precisely to those relations whose spaces of theory-pairs are closed under ultraproducts, generalizing the well-known result for finitary $\TCon$-consequence relations; given that the latter are also known to correspond to those relations whose spaces of theories are `inductive', i.e.\ those that contain the union of upward directed families (cf.~\S 1.3.3~\& 1.3.5 of~\cite{Wojcicki:TLC}), it would of course be only natural to look for a generalization of such inductiveness condition in the context of $\SCon$-consequence.

Finally, it is worth emphasizing that the present investigation has all been done  from the perspective of the smoothest possible generalization of the received (Tarskian-inspired) notion of consequence.  As soon as one entertains the possibility of a notion of compatibility that allows either for gappy or for glutty reasoning, the corresponding notion of theory and the associated space of theories will have to be suitably adapted.  A general framework allowing for non-Tarskian notions of consequence that are characterized by non-canonical valuations referring to more than two logical values is set up in~\cite{BMW}.  Extending our present foray into the land of theories towards covering many-dimensional notions of entailment looks like a natural plan of attack for the future.

\end{document}